\documentclass[11pt]{amsart}
\usepackage{amssymb,amsthm,amsmath,amstext}
\usepackage{mathrsfs}  
\usepackage{bm}        
\usepackage{mathtools} 

\usepackage[left=1.15in,right=1.15in,top=.9in,bottom=.9in]{geometry}

\usepackage[all]{xy}

\theoremstyle{plain}
\newtheorem{theorem}{Theorem}[section]
\newtheorem*{theorem*}{Theorem}
\newtheorem{prop}[theorem]{Proposition}

\newtheorem{cor}[theorem]{Corollary}
\newtheorem{question}[theorem]{Question}

\newtheorem{theoremintro}{Theorem}

\theoremstyle{definition}
\newtheorem{defin}[theorem]{Definition}
\newtheorem{example}[theorem]{Example}

\theoremstyle{remark}
\newtheorem{remark}[theorem]{Remark}

\newcommand{\sheaf}[1]{\mathscr{#1}}
\newcommand{\LL}{\sheaf{L}}
\newcommand{\OO}{\sheaf{O}}
\newcommand{\MM}{\sheaf{M}}
\newcommand{\EE}{\sheaf{E}}
\newcommand{\FF}{\sheaf{F}}
\newcommand{\HH}{\sheaf{H}}
\newcommand{\NN}{\sheaf{N}}

\renewcommand{\AA}{\sheaf{A}}
\newcommand{\BB}{\sheaf{B}}
\newcommand{\PP}{\sheaf{P}}

\DeclareMathOperator{\HHom}{\sheaf{H}\!\mathit{om}}
\DeclareMathOperator{\Hom}{\mathrm{Hom}}
\DeclareMathOperator{\EEnd}{\sheaf{E}\!\mathit{nd}}

\DeclareMathOperator{\MMat}{\sheaf{M}}

\DeclareMathOperator{\Pic}{\mathrm{Pic}}
\DeclareMathOperator{\SSpec}{\mathbf{Spec}}

\newcommand{\Group}[1]{\mathbf{#1}}
\newcommand{\GL}{\Group{GL}}

\newcommand{\PGL}{\Group{PGL}}
\newcommand{\Orth}{\Group{O}}
\newcommand{\SOrth}{\Group{SO}}
\newcommand{\GOrth}{\Group{GO}}

\newcommand{\muu}{\bm{\mu}}

\DeclareMathOperator{\AAut}{\Group{Aut}}

\newcommand{\Dynkin}[1]{\mathsf{#1}}

\newcommand{\Z}{\mathbb Z}

\newcommand{\C}{\mathbb C}

\newcommand{\Gm}{\Group{G}_{\mathrm{m}}}

\newcommand{\inv}{^{-1}}

\newcommand{\dual}{^{\vee}}
\newcommand{\Ldual}{^{\vee\!\LL}}

\newcommand{\pullback}{^{*}}

\newcommand{\op}{^{\mathrm{op}}}

\newcommand{\exterior}{{\textstyle \bigwedge}}
\newcommand{\tensor}{\otimes}

\newcommand{\adj}[1]{\psi_{#1}}

\newcommand{\can}{\mathrm{can}}

\newcommand{\id}{\mathrm{id}}
\newcommand{\vp}{\varphi}

\newcommand{\et}{\mathrm{\acute{e}t}}
\newcommand{\alg}{\text{-}\mathrm{alg}}

\newcommand{\fppf}{\mathrm{fppf}}

\newcommand{\isom}{\cong}

\newcommand{\Het}{H_{\et}}
\newcommand{\Hfppf}{H_{\fppf}}

\newcommand{\discs}{d}
\DeclareMathOperator{\im}{\mathrm{im}}

\newcommand{\linedef}[1]{\textsl{#1}}

\newcommand{\CliffAlg}{\sheaf{C}}
\newcommand{\CliffZ}{\sheaf{Z}}

\renewcommand{\Gm}{\Group{G}_{\mathrm{m}}}
\newcommand{\Pf}{\mathscr{P}\!\mathit{f}}
\newcommand{\pf}{\mathrm{pf}}
\newcommand{\tot}{\mathrm{tot}}
\newcommand{\Wtot}{W_{\tot}}
\newcommand{\Itot}{I_{\tot}}

\newcommand{\Br}{\mathrm{Br}}
\newcommand{\Brtwo}{{}_2\Br}

\newcommand{\cd}{\mathit{cd}}

\newcommand{\JJ}{\sheaf{J}}

\newcommand{\ind}{\mathrm{index}}
\newcommand{\per}{\mathrm{period}}
\newcommand{\PQF}{\mathrm{PQF}}
\newcommand{\Aztwo}{{}_2\mathrm{Az}}

\newcommand{\Nrd}{\mathrm{Nrd}}

\usepackage{hyperref}
\hypersetup{pdftitle={Surjectivity of the total Clifford invariant and Brauer dimension}}
\hypersetup{pdfauthor={Asher Auel}}
\hypersetup{colorlinks=true,linkcolor=blue,anchorcolor=blue,citecolor=blue}

\begin{document}

\vspace*{-0.1cm}

\title[Surjectivity of the total Clifford invariant]{Surjectivity of the
total Clifford invariant\\ and Brauer dimension} 

\author{Asher Auel}
\address{Department of Mathematics\\%
Yale University\\%
10 Hillhouse Ave.\\%
New Haven, CT 06511}
\email{asher.auel@yale.edu}


\begin{abstract}
A celebrated theorem of Merkurjev---that the 2-torsion of the Brauer
group is represented by Clifford algebras of quadratic forms---is in
general false when the base is no longer a field.  The first
counterexamples, when the base is among certain arithmetically subtle
hyperelliptic curves over local fields, were constructed by Parimala,
Scharlau, and Sridharan.  We prove that considering Clifford algebras
of all line bundle-valued quadratic forms, such counterexamples
disappear and we recover Merkurjev's theorem in these cases:\ for any
smooth curve over a local field or any smooth surface over a finite
field, the 2-torsion of the Brauer group is always represented by Clifford algebras of line bundle-valued quadratic forms.
\end{abstract}

\maketitle

\section*{Introduction}

A consequence of Merkurjev's celebrated result
\cite{merkurjev:degree_2}---settling the degree 2 case of the Milnor
conjecture---is that every 2-torsion Brauer class over a field of
characteristic $\neq 2$ is represented by the Clifford algebra of a
quadratic form.  There are many alternate proofs of Merkurjev's
theorem~\cite{arason:proof_Merkurjev_theorem},
\cite{merkurjev:another_proof_norm_2},
\cite{wadsworth:proof_Merkurjev_theorem},
\cite[VIII]{elman_karpenko_merkurjev}, and it retains its status as
one of the great breakthroughs in the theory of quadratic forms in the
second half of the 20th century.

There have been many investigations into the validity of aspects of
the Milnor conjecture over more general rings.  For example, see
\cite[\S3]{guin:homologie_GL_Milnor_K-theory},
\cite{elbaz-vincent_muller-stach},
\cite{hoobler:Merkurjev-Suslin_semilocal_ring}, and
\cite{kerz:Milnor-Chow_homomorphism},
\cite{kerz:Gersten_conjecture_Milnor_K-theory} for a Milnor
$K$-theoretic perspective, \cite{parimala_sridharan:graded_Witt}, 
\cite{morel:proof_milnor}, and \cite{gille:graded_Gersten-Witt} for a
Witt group perspective, and \cite{auel:kyoto} for a survey of results.
In this context, Alex Hahn asked if there exists a commutative ring
$R$ over which the analogue of Merkurjev's theorem doesn't hold, i.e.,
$\Brtwo(R)$ is not represented by Clifford algebras of regular
quadratic forms over $R$. 
The surprising results of Parimala, Scharlau, and Sridharan
\cite{parimala_scharlau:extension},
\cite{parimala_sridharan:graded_Witt},
\cite{parimala_sridharan:nonsurjectivity} show that for a smooth
complete hyperelliptic curve $X$ with a rational point over a local
field of characteristic $\neq 2$, the analogue of Merkurjev's theorem
over $X$ holds if and only if $X$ has a rational theta characteristic
(which can fail to happen). 
These examples are also used to construct affine schemes over which
Merkurjev's theorem does not hold, thus answering Hahn's original
question.

In this work we show that even when Brauer classes of period 2 over a
given scheme $X$ cannot be represented by Clifford algebras of regular
quadratic forms over $X$, they may be represented by Clifford algebras
of regular line bundle-valued quadratic forms. Let $\Wtot(X)$ be the
total Witt group of line bundle-valued quadratic forms (see
\S\ref{subsec:total} for definitions) and let $\Itot^2(X)$ be the
subgroup of line bundle-valued quadratic forms of even rank and
trivial discriminant. We construct (in \S\ref{subsec:total_Clifford})
a natural group homomorphism $e^2 : \Itot^2(X) \to \Brtwo(X)$ with
values in the 2-torsion of the Brauer group of $X$, generalizing the
classical Clifford invariant, and which we call the \linedef{total
Clifford invariant}. A succinct consequence of our main result is the following.

\begin{theoremintro}
\label{thm:A}
Let $X$ be a smooth curve over a local field of characteristic $\neq
2$ or a smooth surface over a finite field of odd characteristic. Then the
total Clifford invariant
$$
e^2 : \Itot^2(X) \to \Brtwo(X)
$$
is surjective.  In other words, every 2-torsion Brauer class on $X$ is
represented by the Clifford algebra of a regular line bundle-valued
quadratic form on $X$.
\end{theoremintro}

In the proof (see \S\ref{sec:surjectivity}), we apply results of
Saltman~\cite{saltman:division_algebra_p-adic_curves} and
Lieblich~\cite{lieblich:transcendence_2} on the Brauer dimension of
function fields of curves over local fields and surfaces over finite
fields, respectively.  Together with a purity result for division
algebras on surfaces (Theorem~\ref{thm:purity}), we reduce the problem
to one concerning Azumaya algebras of degree dividing 4 and index
dividing 2. Then we generalize results of Knus, Ojanguren, Parimala,
Paques, and
Sridharan~\cite{knus_ojanguren_sridharan:quadratic_azumaya},
\cite{knus_paques:rank_4}, \cite{knus:pfaffians_and_quadratic_forms},
\cite{knus_parimala_sridharan:rank_4},
\cite{knus_parimala_sridharan:rank_6_via_pfaffians}, and
\cite{bichsel_knus:values_line_bundles} (also see~\cite[IV~\S
15]{book_of_involutions}), on the exceptional isomorphisms of Dynkin
diagrams $\Dynkin{A}_1^2 = \Dynkin{D}_2$ and $\Dynkin{A}_3 =
\Dynkin{D}_3$, which provide beautiful constructions of line
bundle-valued quadratic forms with specified even Clifford algebras.
In fact, our main result (Theorem~\ref{prop:main2}) applies to any
regular integral scheme $X$ satisfying purity and Brauer dimension
bounded by 2 for algebras of period 2 over the function field.

The verification that the total Clifford invariant is well defined is
no small task, and occupies the bulk of
\S\ref{sec:line}--\ref{sec:invariants}.  The majority of the work goes
into establishing two fundamental algebraic structural results:\ the
Brauer triviality of the even Clifford algebra of a line bundle-valued
metabolic form (Theorem~\ref{thm:metabolic}), generalizing the main
result of \cite{knus_ojanguren:metabolic}; and a formula to compute
the even Clifford algebras and bimodules of orthogonal sums
(Theorem~\ref{thm:Cliff_perp_formula}) leading to a generalization of
the classical fundamental relation in the Brauer group
(Theorem~\ref{thm:Cliff_perp}).  To this end, we use a new direct
tensorial construction of the even Clifford algebra and bimodule (see
\S\ref{subsec:Even_Clifford_algebra}), which offers novel universal
properties (Propositions~\ref{prop:universal_C0} and
\ref{prop:universal_C1}) useful in establishing these results.  These
structural results for line bundle-valued forms are new and are useful
in a variety of contexts.  In particular, they go beyond the author's
previous cohomological construction \cite{auel:clifford} of
Clifford-type invariants.\\[-11pt]

\noindent{\textbf{History.}} 
The notion of a line bundle-valued quadratic form on $X$ appeared
in many different contexts in the early 1970s.
Geyer--Harder--Knebusch--Scharlau
\cite{geyer_harder_knebusch_scharlau} introduced the notion of
symmetric bilinear forms with values in the module of K\"{a}hler
differentials over a global function field.  This notion enables a
consistent choice of local traces in order to generalize residue
theorems to nonrational function fields.  For a smooth complete
algebraic curve $X$, Mumford~\cite{mumford:theta_characteristics}
introduced the notion of locally free $\OO_X$-modules with pairings
taking values in the sheaf of differentials $\omega_X$ to study theta
characteristics.  Kanzaki~\cite{kanzaki:bilinear_module} introduced
the notion of a Witt group of quadratic forms with values in an
invertible module over a commutative ring.
Saltman~\cite[Thm.~4.2]{saltman:azumaya_algebras_with_involution}
showed that involutions
on Azumaya algebras naturally lead to the consideration of line
bundle-valued bilinear forms.

In his thesis, Bichsel~\cite{bichsel:thesis} (reported in
\cite{bichsel_knus:values_line_bundles}) constructed an even Clifford
algebra of a line bundle-valued quadratic form.  This was later used
in~\cite{parimala_sridharan:norms_and_pfaffians},
\cite{balaji_ternary}, and~\cite{voight:characterizing_quaternion}.
Kapranov~\cite[\S4.1]{kapranov:derived} constructed a
\emph{homogeneous} Clifford algebra of a quadratic form---which in
hindsight is related to the \emph{generalized} Clifford algebra of \cite{bichsel_knus:values_line_bundles} or the \emph{graded}
Clifford algebra of~\cite{caenepeel_van_oystaeyen}---to
study the derived category of projective quadrics and quadric
fibrations.  This was further developed by
Kuznetsov~\cite{kuznetsov:quadrics}.  With respect to Clifford
algebras, line bundle-valued quadratic forms behave much like Azumaya
algebras with orthogonal involutions, which do not enjoy a ``full''
Clifford algebra, only an even part together with a bimodule.  In
particular, line bundle-valued quadratic forms have no Clifford
invariant in the classical sense.  The construction of secondary
invariants in \'etale cohomology capturing the even Clifford algebra
of a line bundle-valued quadratic form with fixed discriminant
appeared in \cite{auel:clifford}.  In the
present work, we develop a purely algebraic Clifford invariant for
line bundle-valued quadratic forms with trivial discriminant, taking
values in the 2-torsion of the Brauer group.

\smallskip

{\small\noindent{\textbf{Acknowledgments.}}
The author would like to thank the directors of the
Max-Planck-Institut f\"ur Mathematik in Bonn for providing excellent
working conditions and a wonderfully stimulating environment where
much of this work was accomplished.  A visit at the Forschungsinstitut
f\"ur Mathematik at ETH Z\"urich also proved to be very fruitful.
The author would also personally like to thank B.\ Calm\`es, T.\
Chinburg, M.\ Knus, R.\ Parimala, D.\ Saltman, and V.\ Suresh for many
useful conversations and much encouragement.  Author partially
supported by National Science Foundation MSPRF grant DMS-0903039 and
an NSA
Young Investigator Grant.
}

\section{Line bundle-valued quadratic forms and even Clifford algebras}
\label{sec:line}

Let $X$ be a separated noetherian scheme.  By a \linedef{vector
bundle}, we mean a locally free $\OO_X$-module of constant finite
rank.  Fix a \linedef{line bundle} $\LL$ on $X$, i.e., an invertible
$\OO_X$-module.

\subsection{Line bundle-valued quadratic forms}

A \linedef{(line bundle-valued) symmetric bilinear form} on $X$ is a
triple $(\EE,b,\LL)$, where $\EE$ is a vector bundle on $X$ and $b :
S^2\EE \to \LL$ is an $\OO_X$-module morphism.  A \linedef{(line
bundle-valued) quadratic form} on $X$ is a triple $(\EE,q,\LL)$, where
$\EE$ is a vector bundle on $X$ and $q : \EE \to \LL$ is an
$\OO_X$-homogeneous morphism of degree two such that the associated
morphism $b_q : S^2\EE \to \LL$ defined on sections by $b_q(vw) =
q(v+w) - q(v) - q(w)$ is a symmetric bilinear form. 
We will mostly dispense with the title ``line bundle-valued.''  The
\linedef{rank} of $(\EE,q,\LL)$ is the rank of $\EE$.

A symmetric bilinear form $(\EE,b,\LL)$ is \linedef{regular} if the
canonical adjoint $\adj{b} : \EE \to \HHom(\EE,\LL)$ is an
isomorphism.  A quadratic form $q$ is \linedef{regular} if $b_q$ is
regular. If 2 is assumed invertible on $X$, then we can pass back and
forth between quadratic and symmetric bilinear forms on $X$.

A \linedef{similarity transformation} between symmetric bilinear forms
$(\EE,b,\LL)$ and $(\EE',b',\LL')$ or quadratic forms $(\EE,q,\LL)$
and $(\EE',q',\LL')$ is a pair $(\vp,\lambda)$ consisting of
$\OO_X$-module isomorphisms $\vp : \EE \to \EE'$ and $\lambda : \LL
\to \LL'$ such that 
$b'(\vp(v), \vp(w)) =
\lambda \circ b(v, w)$ or $q'(\vp(v)) = \lambda \circ q(v)$ on
sections, respectively.  A similarity transformation $(\vp,\lambda)$ is an
\linedef{isometry} if $\LL=\LL'$ and $\lambda$ is the identity map.

Denote by $\GOrth(\EE,q,\LL)$ (resp.\ $\Orth(\EE,q,\LL)$) the
presheaf, on the large fppf site $X_{\fppf}$, of similitudes (resp.\
isometries) of a regular quadratic form $(\EE,q,\LL)$.  In fact, this
is a sheaf and is representable by a smooth affine reductive group
scheme over $X$; see~\cite[II.1.2.6,~III.5.2.3]{demazure_gabriel}).
Here we consider reductive group schemes whose fibers are not
necessarily geometrically integral, in contrast to~\cite[XIX.2]{SGA3}.
In particular, the pointed nonabelian cohomology set
$\Hfppf^1(X,\GOrth(\EE,q,\LL))$ is in bijection with the similarity
classes of regular line bundle-valued quadratic forms with the same
rank as $(\EE,q,\LL)$; see \cite[Prop.~1.2]{auel:clifford}.  If $n$ is even
or 2 is invertible on $X$, then the fppf site can be replaced by the
\'etale site.

Define the \linedef{projective similarity} class of a quadratic form
$(\EE,q,\LL)$ to be the set of similarity classes of quadratic forms
$(\NN\tensor\EE,q_{\NN}\tensor q,\NN^{\tensor 2}\tensor\LL)$ ranging
over all regular bilinear forms $(\NN,q_{\NN},\NN^{\tensor 2})$ of
rank 1 on $X$.  In~\cite{balmer_calmes:lax}, this is referred to as a
\linedef{lax-similarity} class.  In their notation, a
\linedef{quadratic alignment} $A = (\NN,\phi)$ between line bundles
$\LL$ and $\LL'$ consists of a line bundle $\NN$ and an $\OO_X$-module
isomorphism $\phi : \NN^{\tensor 2}\tensor \LL \to \LL'$.  A quadratic
alignment induces an equivalence $A^{\circlearrowleft}$ between
categories of $\LL$-valued and $\LL'$-valued quadratic forms (in
particular, an isomorphism $A^{\circlearrowleft} : W(X,\LL) \to
W(X,\LL')$ of Witt groups) defined by $A^{\circlearrowleft}:
(\EE,q,\LL) \mapsto (\NN\tensor\EE, \phi \circ (q_{\NN}\tensor
q),\LL')$, where $q_{\NN}: \NN \to\NN^{\tensor 2}$ is the canonical
squaring form.

\subsection{Even Clifford algebra}
\label{subsec:Even_Clifford_algebra}

In his thesis, Bichsel~\cite{bichsel:thesis} constructs an even
Clifford algebra of a line bundle-valued quadratic form on an affine
scheme.  Alternate constructions are given in
\cite{bichsel_knus:values_line_bundles},
\cite{caenepeel_van_oystaeyen},
and~\cite[{\S4}]{parimala_sridharan:norms_and_pfaffians}, which are
all detailed in \cite[\S1.8]{auel:clifford}.  Inspired by
\cite[II~Lemma~8.1,~\S9]{book_of_involutions}, we now give a direct
tensorial construction.  Let $(\EE,q,\LL)$ be a line bundle-valued
quadratic form on $X$

Define ideals $\JJ_1$ and $\JJ_2$ of the tensor algebra
$T(\EE\tensor\EE\tensor\LL\dual)$ to be locally generated by
\begin{equation}
v\tensor v \tensor f - f(q(v))\cdot 1 \quad \mbox{and} \quad 
u\tensor v \tensor f \tensor v \tensor w \tensor g - 
f(q(v)) \, u \tensor w \tensor g,
\end{equation}
respectively, for sections $u,v,w$ of $\EE$ and $f,g$ of $\LL\dual$.
We define
\begin{equation}
\label{eq:C_0}
\CliffAlg_0(\EE,q,\LL) = T(\EE\tensor\EE\tensor\LL\dual)/(\JJ_1
+ \JJ_2)
\end{equation} 
together with the canonically induced morphism of $\OO_X$-modules
\begin{equation}
\label{eq:C_0_injection}
i : \EE\tensor\EE\tensor\LL\dual \to \CliffAlg_0(\EE,q,\LL),
\end{equation}
which factors through the degree one elements of the tensor algebra.  

Writing the rank as
$n=2m$ or $n=2m+1$, there is a filtration by $\OO_X$-modules
$$
\OO_X = \FF_0 \subset \FF_2 \subset \dotsm \subset \FF_{2m} =\CliffAlg_0(\EE,q,\LL),
$$
where $\FF_{2i}$ is the image of the truncated tensor algebra $T^{\leq
i}(\EE\tensor\EE \tensor \LL\dual)$ in $\CliffAlg_0(\EE,q,\LL)$, for
each $0 \leq i \leq m$.  As
in~\cite[IV~\S1.6]{knus:quadratic_hermitian_forms}, this filtration
has associated graded pieces $\FF_{2i}/\FF_{2(i-1)} \isom
\exterior^{2i} \EE \tensor (\LL\dual)^{\tensor i}$. 
In particular, $\CliffAlg_0(\EE,q,\LL)$ is a locally free
$\OO_X$-algebra of rank $2^{n-1}$.  
By its tensorial construction, the
even Clifford algebra has the following.
 
\begin{prop}[Universal Property of the even Clifford algebra]
\label{prop:universal_C0}
Given an $\OO_X$-algebra $\AA$ and an $\OO_X$-module morphism $j : \EE
\tensor \EE \tensor \LL\dual \to \AA$ such that
$$
j(v\tensor v \tensor f) = f(q(v))\cdot 1 \quad \mbox{and} \quad 
j(u\tensor v \tensor f) \cdot j(v \tensor w \tensor g) = 
f(q(v)) \, j(u \tensor w \tensor g),
$$
then there exists a unique $\OO_X$-algebra homomorphism $\psi :
\CliffAlg_0(\EE,q,\LL) \to \AA$ satisfying $j = \psi \circ i$.  
\end{prop}

A similar universal property for algebras with involution over a field
is stated in \cite[\S3]{mahmoudi:Clifford_algebras}.  The even
Clifford algebra has the following additional properties.

\begin{prop}
\label{prop:properties}
Let $(\EE,q,\LL)$ be a regular quadratic form of rank $n$ on a scheme
$X$.  Write $n=2m$ or $n=2m+1$.
\begin{enumerate}
\item \label{prop:properties.center} If $n$ is odd, $\CliffAlg_0(\EE,q,\LL)$ is a central
$\OO_X$-algebra.  If $n$ is even, the center $\CliffZ(\EE,q,\LL)$
of $\CliffAlg_0(\EE,q,\LL)$ is an \'etale quadratic $\OO_X$-algebra.

\item \label{prop:properties.Azumaya} If $n$ is odd, $\CliffAlg_0(\EE,q,\LL)$ is an Azumaya
$\OO_X$-algebra of degree $2^{m}$.  If $n$ is even,
$\CliffAlg_0(\EE,q,\LL)$ is an Azumaya
$\CliffZ(\EE,q,\LL)$-algebra 
of rank $2^{m-1}$

\item \label{prop:properties.injection} The canonical $\OO_X$-module
morphism $i : \EE \tensor \EE \tensor \LL\dual \to
\CliffAlg_0(\EE,q,\LL)$ is a locally split embedding and there exists a unique
\linedef{canonical involution} $\tau_0 : \CliffAlg_0(\EE,q,\LL) \to \CliffAlg_0(\EE,q,\LL)\op$
satisfying $\tau_0(i(v\tensor w \tensor f)) = i(w \tensor v \tensor f)$
for sections $v,w$ of $\EE$ and $f$ of $\LL\dual$.

\item \label{prop:properties.similarity} Any similarity $(\vp,
\lambda) : (\EE,q,\LL) \to (\EE',q',\LL')$ induces an $\OO_X$-algebra
isomorphism
$$
\CliffAlg_0(\vp, \lambda) : \CliffAlg_0(\EE,q,\LL) \to
\CliffAlg_0(\EE',q',\LL')
$$
satisfying $i(v)\tensor i(w) \tensor f \mapsto i(\vp(v)) \tensor i(\vp(w)) \tensor
f \circ \lambda\inv$ for sections $v,w$ of $\EE$ and $f$ of $\LL\dual$.

\item \label{prop:properties.proj} Any quadratic alignment $A =
(\NN,\phi)$, with $\phi : \NN^{\tensor 2}\tensor \LL \to \LL'$,
induces an $\OO_X$-algebra isomorphism
$$
\CliffAlg_0(A^{\circlearrowleft}) : 
\CliffAlg_0(A^{\circlearrowleft}(\EE,q,\LL))
\to \CliffAlg_0(\EE,q,\LL)
$$
satisfying $i(a\tensor v)\tensor i(b \tensor w) \tensor f \mapsto
i(v) \tensor i(w) \tensor \phi'(a\tensor b\tensor f)$, for sections
$a,b$ of $\NN$, $v,w$ of $\EE$, and $f$ of $\LL'{}\dual$, where
$\phi' : \NN^{\tensor 2} \tensor \LL'{}\dual \to \LL\dual$ is the
isomorphism canonically induced from $\phi$.

\item \label{prop:properties.functorial} For any morphism of schemes
$p : X' \to X$, there is a canonical $\OO_X$-module isomorphism
$$
\CliffAlg_0(p\pullback(\EE,q,\LL)) \to p\pullback\CliffAlg_0(\EE,q,\LL).
$$
\end{enumerate}
\end{prop}
\begin{proof}
Properties \ref{prop:properties.center} and
\ref{prop:properties.Azumaya} are \'etale local
and hence follow from the corresponding properties of the classical
even Clifford algebra (cf.\
\cite[IV~Thm.~2.2.3,~Prop.~3.2.4]{knus:quadratic_hermitian_forms}),
also see \cite[\S3]{bichsel_knus:values_line_bundles}.  Properties
\ref{prop:properties.injection},
\ref{prop:properties.similarity}, and
\ref{prop:properties.proj} are all consequence of
the universal property.  Property
\ref{prop:properties.functorial} is a direct
consequence of the tensorial construction.
\end{proof}

\begin{defin}
Let $(\EE,q,\LL)$ be a quadratic form of even rank on $X$.  We call $f
: Z = \SSpec \CliffZ(\EE,q,\LL) \to X$ the \linedef{discriminant
cover} of $(\EE,q,\LL)$.  If $(\EE,q,\LL)$ is regular, then $f : Z \to
X$ is \'etale quadratic.
\end{defin}

\subsection{Clifford bimodule}
\label{subsec:Clifford_bimodule}

As in the case of central simple algebras with orthogonal involution,
line bundle-valued quadratic forms do not generally enjoy a ``full''
Clifford algebra, of which the even Clifford algebra is the even
degree part.  Inspired by \cite[II~\S9]{book_of_involutions}, we can
directly define the \linedef{Clifford bimodule}
$\CliffAlg_1(\EE,q,\LL)$ of a quadratic form $(\EE,q,\LL)$,
corresponding to the ``odd'' part of the classical Clifford algebra.

The $\OO_X$-module $\EE\tensor T(\EE\tensor\EE\tensor\LL\dual)$ has a
natural right $T(\EE\tensor\EE\tensor\LL\dual)$-module structure
denoted by $\tensor$.  The $\OO_X$-bilinear map $* : (\EE\tensor\EE\tensor\LL\dual) \times
\EE \to \EE\tensor (\EE\tensor\EE\tensor\LL\dual)$ defined
by 
$$
(u \tensor v \tensor f) * w 
=
u \tensor (v \tensor w \tensor f)
$$
for sections $u,v,w$ of $\EE$ and $f$ of $\LL\dual$, induces a left
$T(\EE\tensor\EE\tensor\LL\dual)$-module structure $*$ on $\EE\tensor
T(\EE\tensor\EE\tensor\LL\dual)$, uniquely defined so that it commutes
with the natural right $T(\EE\tensor\EE\tensor\LL\dual)$-module
structure.  
We define
\begin{equation}
\label{eq:C_1}
\CliffAlg_1(\EE,q,\LL) = \EE \tensor T(\EE\tensor \EE \tensor
\LL\dual)/(\EE \tensor \JJ_1 + \JJ_1 * \EE)
\end{equation}
together with the canonically induced morphism of $\OO_X$-modules
\begin{equation}
\label{eq:C_1_injection}
i : \EE \to \CliffAlg_1(\EE,q,\LL),
\end{equation}
which is a locally split embedding.  One immediately checks that $\EE
\tensor \JJ_2 \subset \JJ_1 * \EE$ and $\JJ_2 * \EE \subset \EE
\tensor \JJ_1$, hence $\CliffAlg_1(\EE,q,\LL)$ inherits a
$\CliffAlg_0(\EE,q,\LL)$-bimodule structure.  Denote the right and
left $\CliffAlg_0(\EE,q,\LL)$-module structures by $\cdot$ and $*$, respectively.

Writing the rank as $n=2m$ or $n=2m+1$, there is a filtration
$$
\EE = \FF_1 \subset \FF_3 \subset \dotsm \subset \FF_{2m+1} =\CliffAlg_1(\EE,q,\LL),
$$
where $\FF_{2i+1}$ is the image of the truncation $\EE \tensor T^{\leq
i}(\EE\tensor\EE \tensor \LL\dual)$ in $\CliffAlg_1(\EE,q,\LL)$, for
each $0 \leq i \leq m$.  This filtration has associated graded pieces
$\FF_{2i+1}/\FF_{2i-1} \isom \exterior^{2i+1} \EE \tensor
(\LL\dual)^{\tensor i}$. 
In particular, $\CliffAlg_1(\EE,q,\LL)$ is a locally free
$\OO_X$-module of rank $2^{n-1}$.  By its tensorial construction, the
Clifford bimodule has the following.

\begin{prop}[Universal Property of the Clifford bimodule]
\label{prop:universal_C1}
Given a $\CliffAlg_0(\EE,q,\LL)$-bimodule $\BB$ (with right and
left actions $\cdot$ and $*$) and an
$\OO_X$-module morphism $j : \EE \to \BB$ such that
$$
j(u) \cdot i(v \tensor w \tensor f) = i(u \tensor v \tensor f) * j(w),
$$
for sections $u,v,w$ of $\EE$ and $f$ of $\LL\dual$, there exists a
unique $\CliffAlg_0(\EE,q,\LL)$-bimodule morphism $\psi :
\CliffAlg_1(\EE,q,\LL) \to \BB$ satisfying $j = \psi \circ i$.
\end{prop}

The Clifford bimodule has the following additional properties.

\begin{prop}
\label{prop:properties_C_1}
Let $(\EE,q,\LL)$ be a regular quadratic form on a scheme $X$.
\begin{enumerate}  \setlength{\itemsep}{2pt}
\item \label{properties_C_1.invertible} The Clifford bimodule
$\CliffAlg_1(\EE,q,\LL)$ is invertible as a (left or right)
$\CliffAlg_0(\EE,q,\LL)$-module. 

\item \label{properties_C_1.semilinear} If $n$ is even, then the
action of $\CliffZ(\EE,q,\LL)$ on $\CliffAlg_1(\EE,q,\LL)$ satisfies
$x \cdot z = \iota(z) * x$ for sections $z$ of $\CliffZ(\EE,q,\LL)$
and $x$ of $\CliffAlg_1(\EE,q,\LL)$, where $\iota$ is the nontrivial
$\OO_X$-automorphism of $\CliffZ(\EE,q,\LL)$.

\item \label{properties_C_1.m} 
There is a canonical isomorphism
$$
m : \CliffAlg_1(\EE,q,\LL) \otimes_{\CliffAlg_0(\EE,q,\LL)}
\CliffAlg_1(\EE,q,\LL) \to \CliffAlg_0(\EE,q,\LL)\tensor_{\OO_X}\LL 
$$
of $\CliffAlg_0(\EE,q,\LL)$-bimodules satisfying 
$
m(i(v)\tensor i(v)) = 1 \tensor q(v)
$ 
for a section $v$ of $\EE$.

\item \label{properties_C_1.similarity}
Any similarity transformation $(\vp, \lambda) : (\EE,q,\LL) \to
(\EE',q',\LL')$ induces an $\OO_X$-module isomorphism
$$
\CliffAlg_1(\vp, \lambda) : \CliffAlg_1(\EE,q,\LL) \to
\CliffAlg_1(\EE',q',\LL').
$$
that is $\CliffAlg_0(\vp,\lambda)$-semilinear with respect to the
bimodule structure.

\item \label{properties_C_1.alignment} Any quadratic alignment $A = (\NN,\phi)$, with $\phi :
\NN^{\tensor 2}\tensor \LL \to \LL'$, induces an $\OO_X$-module
isomorphism
$$
\CliffAlg_1(A^{\circlearrowleft}) :
\CliffAlg_1(A^{\circlearrowleft}(\EE,q,\LL)) \to \NN \tensor \CliffAlg_1(\EE,q,\LL)
$$
that is $\CliffAlg_0(A^{\circlearrowleft})$-semilinear with respect to
the bimodule structure.

\item \label{properties_C_1.functorial} For any morphism of schemes $p
: X' \to X$, there is a canonical $\OO_X$-module isomorphism
$$
\CliffAlg_1(p\pullback(\EE,q,\LL)) \to p\pullback\CliffAlg_1(\EE,q,\LL) .
$$
\end{enumerate}
\end{prop}
\begin{proof}
For simplicity, we write $\CliffAlg_0 = \CliffAlg_0(\EE,q,\LL)$ and
$\CliffAlg_1 = \CliffAlg_1(\EE,q,\LL)$.  For
\ref{properties_C_1.invertible}, since $q$ is fiberwise nonzero,
Zariski locally there exists a line subbundle $\NN \subset \EE$ such
that $q|_\NN$ is regular.  Then as in the classical case (see
\cite[IV~Prop.~7.5.2]{knus:quadratic_hermitian_forms}), $\NN$ locally
generates $\CliffAlg_1$ over $\CliffAlg_0$ as a right or left module.

For \ref{properties_C_1.semilinear}, this is a local question and
hence follows from
\cite[IV~Prop.~4.3.1(4)]{knus:quadratic_hermitian_forms}.  For
\ref{properties_C_1.m}, we will define a
$\CliffAlg_0(\EE,q,\LL)$-bimodule morphism $\psi_m :
\CliffAlg_1(\EE,q,\LL) \to \HHom_{\CliffAlg_0}(\CliffAlg_1,\CliffAlg_0
\tensor\LL)$, where $\HHom_{\CliffAlg_0}$ denotes the sheaf of right
$\CliffAlg_0$-module homomorphisms (here $\LL$ is acted trivially on).
Then $m$ will be the $\CliffAlg_0$-bimodule map with adjoint $\psi_m$.
To this end, for each section $v$ of $\EE$, we define a section $m_v$
of $\HHom_{\CliffAlg_0}(\CliffAlg_1,\CliffAlg_0 \tensor\LL) \isom
\HHom_{\CliffAlg_0}(\CliffAlg_1 \tensor \LL\dual
\tensor\LL,\CliffAlg_0 \tensor\LL)$ by applying the universal property
to the map $w \tensor f \tensor l \mapsto i(v \tensor w \tensor f)
\tensor l$ for a section $w$ of $\EE$, $f$ of $\LL\dual$, and $l$ of
$\EE$.  Then applying the universal property to the map 
defined by $v \mapsto m_v$, yields the required $\psi_m$.  Finally,
$m$ is an isomorphism by \ref{properties_C_1.invertible}, since it's a
nontrivial map of invertible $\CliffAlg_0$-bimodules.

Properties \ref{properties_C_1.similarity} and
\ref{properties_C_1.alignment} are consequence of the universal
property (cf.\ \cite[Prop.~2.6]{bichsel:thesis} and
\cite[Lemma~3.3]{bichsel_knus:values_line_bundles}).  Property
\ref{properties_C_1.functorial} is a direct consequence of the
tensorial construction.
 \end{proof}

\subsection{Metabolic forms}
\label{subsec:Metabolic_forms}

A quadratic form $(\EE,q,\LL)$ of rank $n=2m$ on $X$ is
\linedef{metabolic} if there exists a locally direct summand $\FF \to
\EE$ of rank $m$ such that the restriction of $q$ to $\FF$ is zero.
Any choice of such $\PP$ is a \linedef{lagrangian}. The class of
hyperbolic forms is the main example.

\begin{example}
\label{ex:hyperbolic}
For any vector bundle $\PP$ of rank $m$ and any line bundle $\LL$ the
\linedef{($\LL$-valued) hyperbolic quadratic form} $H_{\LL}(\PP)$ has
underlying $\OO_X$-module $\HHom(\PP,\LL)\oplus \PP$ and is given by
$t+v \mapsto t(v)$ on sections.  Here, $\PP$ and $\HHom(\PP,\LL)$ are
lagrangians.

We now proceed to compute the even Clifford algebra and Clifford
bimodule of a hyperbolic form, which will be necessary for us later.
Given an $\OO_X$-module morphism $t : \PP \to \LL$, for each $i \geq
0$ we define
$$
d_t^{(i)} : \exterior^{i+1}\PP \to \exterior^{i}\PP
\tensor \LL
$$ 
inductively by $d_t^{(i)}(v \wedge x) = x \tensor t(v) + x \wedge
d_t^{(i-1)}(x)$ for sections $v$ of $\PP$ and $x$ of $\exterior^i\PP$,
cf.\ \cite[\S2]{auel:euler_four}.  Under the identification
$\exterior^0\PP=\OO_X$, we set $d_t^{(0)}=t$.  Defining
$$
\exterior^{+}_\LL \PP = \bigoplus_{i=0}^{\lfloor m/2 \rfloor}\exterior^{2i} \PP\tensor (\LL\dual)^{\tensor i}, \qquad 
\exterior^{-}_\LL \PP = \bigoplus_{i=0}^{\lfloor (m-1)/2 \rfloor}\exterior^{2i+1} \PP \tensor(\LL\dual)^{\tensor i}.
$$
there are induced $\OO_X$-module morphisms
$$
d_t^+ : \exterior^+_\LL \PP \to \exterior^-_\LL \PP, \qquad
d_t^- : \exterior^-_\LL \PP \to \exterior^+_\LL \PP \tensor\LL.
$$ 
Also, for each global section $v$ of $\PP$, left
wedging defines $\OO_X$-module morphisms
$$
l_v^+ : \exterior^+_\LL \PP \to \exterior^-_\LL \PP, \qquad
l_v^- : \exterior^-_\LL \PP \to \exterior^+_\LL \PP \tensor\LL.
$$
One immediately checks that the maps
\begin{align*}
H_\LL(\PP) \tensor H_\LL(\PP) \tensor \LL\dual & {} \to \EEnd\left(
\exterior^{+}_\LL \PP \right) \times
\EEnd\left( \exterior^{-}_\LL \PP  \right) \\
(t + v) \tensor (s + w) \tensor f & {} \mapsto (\id\tensor
f)(d_t^-\circ d_s^+ + d_t^- \circ l_w^+ + l_v^- \circ d_s^+ + l_v^- \circ l_w^+) \\
 & {} \qquad + ( d_t^+ \tensor f \circ d_s^- +
d_t^+ \tensor f \circ l_w^- + l_v^+ \tensor f \circ d_s^- + l_v^+
\tensor f \circ l_w^-)
\end{align*}
and
\begin{align*}
H_\LL(\PP) & {} \to \HHom\bigl(\exterior^{+}_\LL
\PP, \exterior^{-}_\LL \PP\bigr) \oplus \HHom\bigl(\exterior^{-}_\LL
\PP, \exterior^{+}_\LL \PP\bigr) \tensor \LL \\
t+v  & {} \mapsto (d_t^+ + l_v^+) +
(d_t^- + l_v^-).
\end{align*}
satisfy the universal properties of the even Clifford algebra and
Clifford bimodule, hence induce a canonical $\OO_X$-algebra morphism
$$
\Phi_0 : \CliffAlg_0(H_{\LL}(\PP)) \to
\EEnd\left( \exterior^{+}_\LL \PP \right) \times 
\EEnd\left( \exterior^{-}_\LL \PP  \right) 
$$
and a canonical $\OO_X$-module morphism
$$
\Phi_1 : \CliffAlg_1(H_{\LL}(\PP)) \to \HHom\bigl(\exterior^{+}_\LL
\PP, \exterior^{-}_\LL \PP\bigr) \oplus \HHom\bigl(\exterior^{-}_\LL
\PP, \exterior^{+}_\LL \PP\bigr) \tensor \LL
$$
transporting, via the morphism $\Phi_0$, the $\CliffAlg_0(H_\LL(\PP))$-bimodule
structure to the evident composition
$\EEnd(\exterior_\LL^+\PP)\times\EEnd(\exterior_\LL^-\PP)$-bimodule
structure.  Zariski locally, $\Phi_0$ and $\Phi_1$ agree with the
restriction of the classical isomorphism $\CliffAlg(H_{\OO_X}(\PP))
\isom \EEnd(\exterior\PP)$ (see
\cite[IV~Prop.~2.1.1]{knus:quadratic_hermitian_forms}) to the even and
odd components of the Clifford algebra, hence $\Phi_0$ and $\Phi_1$
are isomorphisms.

We point out that $\CliffZ(H_{\LL}(\PP)) \isom \OO_X \times \OO_X$ is
the split \'etale quadratic algebra.
\end{example}

The formula for the even Clifford algebra of a hyperbolic form given
in Example~\ref{ex:hyperbolic} does not persist to (nonsplit)
metabolic quadratic forms, a phenomenon already apparent when
$\LL=\OO_X$; see \cite{knus_ojanguren:metabolic}.  However, the main
result of this section is that $\CliffAlg_0$ is still a product of
split Azumaya algebras.

\begin{theorem}
\label{thm:metabolic}
Let $(\EE,q,\LL)$ be a metabolic quadratic form of rank $n=2m$ on a
scheme $X$. 
Any choice of lagrangian $\FF \to \EE$ induces a natural choice of
vector bundles $\MM^+$ and $\MM^-$ of rank $2^{m-1}$, an
$\OO_X$-algebra isomorphism
$$
\Phi_0 : \CliffAlg_0(\EE,q,\LL) \isom \EEnd(\MM^+) \times \EEnd(\MM^-),
$$
and an $\OO_X$-module isomorphism
$$
\Phi_1 : \CliffAlg_1(\EE,q,\LL) \isom \HHom(\MM^+,\MM^-) \oplus \HHom(\MM^-,\MM^+)\tensor\LL
$$
transporting, via $\Phi_0$, the $\CliffAlg_0(\EE,q,\LL)$-bimodule
structure to the evident composition $\EEnd(\MM^+) \times
\EEnd(\MM^-)$-bimodule structure.
\end{theorem}
\begin{proof}
We generalize the proof from
Knus--Ojanguren~\cite{knus_ojanguren:metabolic} to the line
bundle-valued setting.  On the category of vector bundles, write
$(-)\Ldual$ for the functor $\HHom(-,\LL)$ and $\can_\LL$ for the
canonical isomorphism of functors $\id \to ((-)\Ldual)\Ldual$.

Let $\PP$ be a vector bundle of rank $m \geq 1$ and $H_\LL(\PP)$ be
the corresponding $\LL$-valued hyperbolic form.  Denote by $\gamma_0 :
\Orth\bigl(H_\LL(\PP)\bigr) \to \AAut_{\OO_X\alg}\bigl(
\CliffAlg_0(H_\LL(\PP))\bigr)$ the homomorphism induced by
Proposition~\ref{prop:properties}\ref{prop:properties.similarity}.
Restricting $\gamma_0$ to the center yields the \linedef{Dickson}
homomorphism $\Delta : \Orth\bigl(H_\LL(\PP)\bigr) \to
\AAut_{\OO_X\alg}\bigl( \CliffZ(H_\LL(\PP)) \bigr) = \Z/2\Z$ of group schemes,
cf.\ \cite[\S1.9]{auel:clifford}.  Its kernel is the special
orthogonal group scheme
$\SOrth\bigl(H_\LL(\PP)\bigr)$.  Under the identification
$\CliffAlg_0(H_\LL(\PP)) = \EEnd\bigl(\exterior^+_\LL\PP\bigr) \times
\EEnd\bigl(\exterior^-_\LL\PP\bigr)$ of Example~\ref{ex:hyperbolic},
we have that $\gamma_0$ restricts to a homomorphism
$$
\gamma_0 : \SOrth\bigl(H_\LL(\PP)\bigr) \to
\AAut_{\CliffZ\alg}\bigl(\CliffAlg_0(H_\LL(\PP))\bigr) \isom \PGL\bigl(\exterior^+_\LL\PP\bigr) \times
\PGL\bigl(\exterior^-_\LL\PP\bigr).
$$  

Similarly, denote by $\gamma_1 : \Orth\bigl(H_\LL(\PP)\bigr) \to
\AAut_{\OO_X-\mathrm{mod}}\bigl( \CliffAlg_1(H_\LL(\PP)) \bigr)$ the homomorphism induced by
Proposition~\ref{prop:properties_C_1}\ref{properties_C_1.similarity}.
Under the identification of $\CliffAlg_1(H_\LL(\PP))$ with the vector bundle
$\HHom\left(\exterior^{+}_\LL \PP, \exterior^{-}_\LL \PP\right) \oplus
\HHom\left(\exterior^{-}_\LL \PP, \exterior^{+}_\LL \PP\right) \tensor
\LL$ of Example~\ref{ex:hyperbolic}, we have that $\gamma_1$ restricts
to a homomorphism
$$
\gamma_1 : \SOrth\bigl(H_\LL(\PP)\bigr) \to
\AAut_{\CliffZ}\bigl(\CliffAlg_1(H_\LL(\PP))\bigr) \isom \GL(\HH^+) \times \GL(\HH^-)
$$  
where we write $\HH^+ = \HHom\bigl(\exterior^{+}_\LL\PP, \exterior^{-}_\LL
\PP\bigr)$ and $\HH^- = \HHom\bigl(\exterior^{-}_\LL \PP,
\exterior^{+}_\LL \PP \tensor \LL\bigr)$.

The parabolic subgroup $\SOrth\bigl(H_\LL(\PP),\PP\bigr) \subset
\SOrth\bigl(H_\LL(\PP)\bigr)$ of isometries preserving $\PP$ has the
following block description
$$
\SOrth\bigl(H_\LL(\PP),\PP\bigr) (U) = \left\{
\begin{pmatrix}
(\alpha\Ldual)\inv & \beta \\
0 & \alpha
\end{pmatrix}
 \; : \;
\beta\Ldual\, \can_\LL\, \alpha ~ \text{is alternating}
\;
\right\}
$$
where for each $U \to X$ and each $\alpha \in \Hom(\PP|_U,\PP|_U)$ and
$\beta \in \Hom(\PP|_U,\PP|_U\Ldual)$, we consider $\beta\Ldual \,
\can_\LL \, \alpha : \PP_U \to \PP|_U\Ldual$ as the adjoint of an
$\LL|_U$-valued bilinear form.

We use an $\LL$-valued version of Bourbaki's tensor operations for
even Clifford algebras, cf.\ \cite[Thm.~2.2]{balaji_ternary}.  In
particular, under the canonical identifications
$\CliffAlg_0(\PP,0,\LL) = \exterior_\LL^+\PP$ and
$\CliffAlg_1(\PP,0,\LL) = \exterior_\LL^-\PP$, there exist
homomorphisms of sheaves of groups
$$
\Psi^\pm : \HHom\bigl(\exterior^2\PP,\LL\bigr) \to 
\GL\bigl( \exterior^\pm_\LL\PP \bigr)
$$
satisfying the following properties:
\begin{equation}
\label{eq:Phi_prop}
\Psi^\pm({b \circ \wedge^2\vp}) = \wedge_\LL^\pm(\vp)\inv
\; \Psi^\pm(b)\, \wedge_\LL^\pm(\vp)
\end{equation} 
for each alternating form $b : \exterior^2\PP \to \LL$ and each $\vp
\in \GL(\PP)$; and
\begin{equation}
\label{eq:Phi_bilinear}
\psi_b=\psi_{b'} \quad \Rightarrow \quad \Psi^\pm(b) =
\Psi^\pm({b'})
\end{equation}
where $\psi_b : \PP \to \PP\Ldual$ is the adjoint map to the
alternating form $b : \exterior^2\PP \to \LL$.  By
\eqref{eq:Phi_bilinear}, we can write $\Psi^\pm(\psi)$ in place of
$\Psi^\pm(b)$ for any $\OO_X$-module morphism $\psi : \PP \to
\PP\Ldual$ that is adjoint to an alternating form $b : \exterior^2\PP
\to \LL$.

With this in hand, we define maps
\begin{align*}
\rho^\pm : \SOrth\bigl(H_\LL(\PP),\PP\bigr) & {} \to \GL\bigl(\exterior_\LL^\pm\PP\bigr)\\
 \begin{pmatrix}
(\alpha\Ldual)\inv & \beta \\
0 & \alpha
\end{pmatrix}
 & {} \mapsto \wedge_\LL^\pm(\alpha)\, \Psi^\pm({\alpha\Ldual \beta})
\end{align*}
which we now proceed to verify are well defined homomorphisms.
Consider the Levi decomposition $\SOrth\bigl(H_\LL(\PP),\PP\bigr)
=\Group{M}\Group{N}=\Group{N}\Group{M}$ given explicitly by
$$
\begin{pmatrix}
(\alpha\Ldual)\inv & \beta \\
0 & \alpha
\end{pmatrix}
=
\begin{pmatrix}
(\alpha\Ldual)\inv & 0 \\
0 & \alpha
\end{pmatrix}
\begin{pmatrix}
1 & \alpha\Ldual \beta \\
0 & 1
\end{pmatrix}
=
\begin{pmatrix}
1 & \beta \alpha\inv \\
0 & 1
\end{pmatrix}
\begin{pmatrix}
(\alpha\Ldual)\inv & 0 \\
0 & \alpha
\end{pmatrix}
$$
and note that $\alpha\Ldual \beta$ (being the transpose of
$\beta\Ldual\,\can_\LL\,\alpha$) is adjoint to an alternating form,
say $b : \exterior^2 \PP \to \LL$.  Then $\beta\alpha\inv$ is adjoint
to the alternating form $b \circ \wedge^2\alpha\inv$, since we can
write $\beta\alpha\inv = (\alpha\inv)\Ldual(\alpha\Ldual \beta)
\alpha\inv$.  Hence by \eqref{eq:Phi_prop}, $\rho^\pm$ is also given
by $\Psi^\pm(\beta\alpha\inv) \wedge^\pm_\LL(\alpha)$.  Since
$\rho^\pm$ is based on, and independent of, the Levi decomposition
order, it is a well defined group scheme homomorphism.

Denoting by $\rho_0 =\rho^+ \times \rho^- :
\SOrth\bigl(H_\LL(\PP),\PP\bigr) \to \GL\bigl(\exterior^+_\LL\PP\bigr)
\times \GL\bigl(\exterior^-_\LL\PP\bigr)$, consider the diagram
$$
\xymatrix{
\SOrth\bigl(H_\LL(\PP),\PP\bigr)\ar[d]_{\rho_0} \ar[r] & \SOrth\bigl(H_\LL(\PP)\bigr) \ar[d]^{\gamma_0}\\
\GL\bigl(\exterior^+_\LL\PP\bigr) \times \GL\bigl(\exterior^-_\LL\PP\bigr) \ar[r] & \PGL\bigl(\exterior^+_\LL\PP\bigr) \times \PGL\bigl(\exterior^-_\LL\PP\bigr)
}
$$
of group schemes, where the horizontal arrows are the obvious ones.
The fiber of this diagram over any point of $X$ is isomorphic to the
restriction, to the special orthogonal group and even Clifford
algebra, of the corresponding commutative diagram of orthogonal groups
and (full) Clifford algebras in \cite[Thm.]{knus_ojanguren:metabolic}
(cf.\ \cite[IV~Prop.~2.4.2]{knus:quadratic_hermitian_forms}).  Hence
the diagram commutes over $X$.

We now consider the induced commutative diagram of pointed nonabelian
cohomology sets:\ $\Het^1\bigl(X,\SOrth(H_\LL(\PP)))$ is in bijection with
the set of similarity classes of $\LL$-valued quadratic forms
$(\EE,q,\LL)$ of rank $2m$ together with an \linedef{orientation}
isomorphism $\zeta : \CliffZ(\EE,q,\LL) \isom \OO_X\times\OO_X$ (cf.\
\cite[Prop.~1.15]{auel:clifford}); $\Het^1(X,\SOrth(H_\LL(\PP),\PP))$
is in bijection with the set of similarity classes of metabolic
$\LL$-valued quadratic forms $(\EE,q,\LL)$ of rank $n=2m$ together
with a choice of lagrangian; $\Het^1(X,\GL(\exterior_\LL^\pm\PP))$ is
in bijection with the set of isomorphism classes of vector bundles
$\MM^\pm$ of rank $2^{m-1}$; $\Het^1(X,\PGL(\exterior_\LL^\pm\PP))$ is
in bijection with the set of isomorphism classes of Azumaya algebras
of degree $2^{m-1}$; the induced map
$$
\Het^1\bigl(X,\SOrth\bigl(H_\LL(\PP),\PP\bigr)\bigr) \to
\Het^1
\bigl(X,\SOrth\bigl(H_\LL(\PP)\bigr)\bigr)
$$ 
replaces the choice of lagrangian by the orientation it canonically
induces (cf.\ \cite[Lemma~1.14]{auel:clifford}); the map induced by
$\gamma_0$ takes an oriented quadratic form to its even Clifford
algebra together with the splitting of its center induced by the
orientation; the map induced by $\rho_0$ takes a metabolic quadratic
form of rank $n=2m$ together with a choice of lagrangian to a pair of
vector bundles $\MM^+$ and $\MM^-$ of rank $2^{m-1}$; the induced map
$$
\Het^1\bigl(X,\GL\bigl(\exterior^\pm_\LL\PP\bigr)\bigr) \to
\Het^1\bigl(X,\PGL\bigl(\exterior^\pm_\LL\PP\bigr)\bigr)
$$
takes a vector bundle $\MM^\pm$
to the Azumaya algebra $\EEnd(\MM^\pm)$.  Chasing the diagram around
shows that if $(\EE,q,\LL)$ is a metabolic quadratic form of rank
$2m$, then $\CliffAlg_0(\EE,b,\LL)$ is isomorphic to $\EEnd(\MM^+)
\times \EEnd(\MM^-)$ for vector bundles $\MM^+$ and $\MM^-$ of rank
$2^{m-1}$ on $X$.

To identify the Clifford bimodule, consider the diagram
$$
\xymatrix{
\SOrth\bigl(H_\LL(\PP),\PP\bigr)\ar[d]_{\rho_0} \ar[r] & \SOrth\bigl(H_\LL(\PP)\bigr) \ar[d]^{\gamma_1}\\
\GL\bigl(\exterior^+_\LL\PP\bigr) \times
\GL\bigl(\exterior^-_\LL\PP\bigr) \ar[r]^(.53){c} & 
\GL(\HH^+) \times \GL(\HH^-)
}
$$
of group schemes, where the top horizontal arrow is the canonical
one, and $c$ is the evident homomorphism defined by compositions.
The fiber of this diagram over any point of $X$ is isomorphic to the
restriction, to the special orthogonal group and odd part of the
Clifford algebra, of the corresponding commutative diagram of
orthogonal groups and (full) Clifford algebras in
\cite[Thm.]{knus_ojanguren:metabolic} (cf.\
\cite[IV~Prop.~2.4.2]{knus:quadratic_hermitian_forms}).  Hence the
diagram commutes over $X$.

As above, we consider the induced commutative diagram of pointed
nonabelian cohomology sets:\ the map induced by $\gamma_1$ takes an
oriented quadratic form to its Clifford bimodule, together with a
direct sum decomposition stable under the action of the center; the
map induced by $c$ takes a pair of vector bundles $\MM^+$ and $\MM^-$
of rank $2^{m-1}$ to $\HHom(\MM^+,\MM^-) \oplus
\HHom(\MM^-,\MM^+\tensor\LL)$.  Chasing the diagram around gives the
stated identification.  The compatibility of the bimodule structures
can then be checked locally.
\end{proof}

\subsection{Orthogonal sums}
\label{subsec:Orthogonal_sums}

We will also need an orthogonal sum formula for the even Clifford
algebra.  Let $(\EE,q,\LL)$ and $(\EE',q',\LL)$ be quadratic forms
over a scheme $X$ and denote by
$$
i_0 : \EE\tensor\EE\tensor\LL\dual \to \CliffAlg_0(q), \qquad
i_0' : \EE'\tensor\EE'\tensor\LL\dual \to \CliffAlg_0(q'),
$$
$$
i_1 : \EE \to \CliffAlg_1(q), \qquad
i_1' : \EE' \to \CliffAlg_1(q'),
$$
the canonical $\OO_X$-module morphisms \eqref{eq:C_0_injection} and
\eqref{eq:C_1_injection}, respectively.

We define an $\OO_X$-algebra structure on
$\CliffAlg_0(q)\tensor\CliffAlg_0(q') \oplus
\CliffAlg_1(q)\tensor\CliffAlg_1(q')\tensor \LL\dual$ as follows: 
by multiplication in $\CliffAlg_0$ (for products between elements of
the first summand), by the $\CliffAlg_0$-bimodule action $\CliffAlg_1$
(between elements of the first and second summands), and by the
multiplication map $m : \CliffAlg_1 \tensor \CliffAlg_1 \to
\CliffAlg_0\tensor \LL$ in
Proposition~\ref{prop:properties_C_1}\ref{properties_C_1.m} (between
elements of the second summand) followed by evaluation with
$\LL\dual$.  One can check that the map
\begin{align*}
(\EE \oplus \EE')\tensor(\EE \oplus \EE')\tensor\LL\dual 
& {} \to
\CliffAlg_0(q)\tensor\CliffAlg_0(q') \oplus
\CliffAlg_1(q)\tensor\CliffAlg_1(q')\tensor \LL\dual \\
(v+v')\tensor(w+w')\tensor f 
& {} \mapsto 
\bigl( 
i_0(v\tensor w \tensor f) \tensor 1 + 1 \tensor
i_0'(v'\tensor w' \tensor f) 
\bigr) \\
& {} \qquad \qquad + 
\bigl(
i_1(v)\tensor i_1'(w') \tensor f - 
i_1(w)\tensor i_1'(v') \tensor f
\bigr)
\end{align*}
satisfies the universal property of the even Clifford algebra, hence
induces an $\OO_X$-algebra morphism $\CliffAlg_0(q \perp q') \to
\CliffAlg_0(q)\tensor\CliffAlg_0(q') \oplus
\CliffAlg_1(q)\tensor\CliffAlg_1(q')\tensor \LL\dual$.  Via this
morphism, there is an induced $\CliffAlg_0(q\perp q')$-bimodule
structure on $\CliffAlg_0(q)\tensor\CliffAlg_1(q') \oplus
\CliffAlg_1(q)\tensor\CliffAlg_0(q')$, and one can check that the map
\begin{align*}
\EE \oplus \EE' & {} \to \CliffAlg_0(q)\tensor\CliffAlg_1(q')
\oplus \CliffAlg_1(q)\tensor\CliffAlg_0(q') \\
v+v' & {} \mapsto  i_1(v) \tensor 1  +  1 \tensor i_1'(v')
\end{align*}
satisfies the universal property of the Clifford bimodule.  

\begin{theorem}
\label{thm:Cliff_perp_formula}
Let $(\EE,q,\LL)$ and $(\EE',q',\LL)$ be quadratic forms
over a scheme $X$.
Then the $\OO_X$-algebra morphism
\begin{equation}
\label{eq:even_C_0_perp_isom}
\CliffAlg_0(q \perp q') \to
\CliffAlg_0(q)\tensor\CliffAlg_0(q') \oplus
\CliffAlg_1(q)\tensor\CliffAlg_1(q')\tensor \LL\dual
\end{equation}
and the $\CliffAlg_0(q \perp q')$-bimodule morphism
\begin{equation}
\label{eq:even_C_1_perp_isom}
\CliffAlg_1(q \perp q') \to \CliffAlg_0(q)\tensor\CliffAlg_1(q')
\oplus \CliffAlg_1(q)\tensor\CliffAlg_0(q'),
\end{equation}
induced from the universal properties, are isomorphisms. 
\end{theorem}
\begin{proof}
Locally, when $\LL$ is trivial, these maps agree with their classical
counterparts (cf.\
\cite[IV~Thm.~1.3.1]{knus:quadratic_hermitian_forms}) and hence are
isomorphisms.
\end{proof}

\section{Total Witt groups and total classical invariants}
\label{sec:invariants}

In this section, we define the notion of total Witt groups and
construct the total classical cohomological invariants on these
groups.

\subsection{Total Witt groups}
\label{subsec:total}

One must be careful when working with ``total'' Witt groups.  Fix a
scheme $X$ and denote by $W(X,\LL)$ the (quadratic) Witt group of
regular $\LL$-valued quadratic forms modulo metabolic forms on $X$.
We usually write $W(X) = W(X,\OO_X)$.  Every Witt class can be
represented by a regular quadratic form, see~\cite[\S5]{knebusch}.

We also fix a set $P$ of line bundle representatives of the quotient
group $\Pic(X)/2$.  With respect to this choice, we define the
\linedef{total (quadratic) Witt group} $\Wtot(X) = \bigoplus_{\LL \in
P} W(X,\LL)$.  While the abelian group $\Wtot(X)$ is only well defined
up to non-canonical isomorphism depending on our choice of $P$, the
cohomological invariants we consider will not depend on such choices.
Definition~\ref{def:system} makes this precise.

Most importantly, we will not consider any ring structure on
$\Wtot(X)$ and thus will not need to descend into the subtle
considerations of~\cite{balmer_calmes:lax}.  

\begin{defin}
\label{def:system}
Fix an abelian group $H$ and group homomorphisms $e_{\LL} : W(X,\LL)
\to H$ for each line bundle $\LL$.  We say that the system
$\{e_{\LL}\}$ is a \linedef{system of projective similarity class
invariants} if for any quadratic alignment $A = (\NN,\phi)$ between
line bundles $\LL$ and $\LL'$, there is a commutative diagram
$$
\xymatrix{
W(X,\LL)\ar[d]_{A^{\circlearrowleft}}\ar[r]^(.65){e_{\LL}} & H \ar@{=}[d]\\
W(X,\LL') \ar[r]^(.65){e_{\LL'}} & H
}
$$ 
of abelian groups.   
One could axiomatize this notion using the language of morphisms of
functors, together with a compatibility condition with respect to
quadratic alignments.

Given a system $\{e_{\LL}\}$ of projective similarity class
invariants, the combined homomorphism
$$
e = \oplus_{\LL \in P}\, e_{\LL}: \Wtot(X) = {\textstyle\bigoplus_{\LL \in P}} W(X,\LL) \to H
$$
is well defined and independent of the choice $P$ of representatives
of $\Pic(X)/2$.  We call $e$ the \linedef{total invariant} associated
to the system $\{e_\LL\}$ of projective similarity class invariants.
\end{defin}

A reader who is unhappy with this formalism may, for example, simply
replace the statement ``$e : \Wtot(X) \to H$ is surjective'' by the
equivalent statement ``for each $h \in H$, there exists a line bundle
$\LL$ and a class $q \in W(X,\LL)$, such that $e_{\LL}(q) = h$''.

\subsection{Rank modulo 2}

For each line bundle $\LL$, the rank modulo 2 defines a homomorphism
$e^0_\LL : W(X,\LL) \to \Het^0(X,\Z/2\Z)$.  Then $\{e^0_\LL\}$ is a
system of projective similarity class invariants and there is a
\linedef{total rank modulo 2} homomorphism
$$
e^0 : \Wtot(X) \to \Het^0(X,\Z/2\Z).
$$
Denote by $I^1(X,\LL) \subset W(X,\LL)$ the kernel of $e^0_\LL$ 
and by $\Itot^1(X) = \bigoplus_{\LL \in P} I^1(X,\LL)$.  Note that if
$\LL$ is not a square in $\Pic(X)$ then $I^1(X,\LL)=W(X,\LL)$, cf.\
\cite[Lemma~1.6]{auel:clifford}. Thus $e^0$ has kernel $\Itot^1(X)$.

\subsection{Total discriminant}

Recall from
Proposition~\ref{prop:properties}\ref{prop:properties.center} that the
center $\CliffZ(\EE,q,\LL)$ of the even Clifford algebra
$\CliffAlg_0(\EE,q,\LL)$ of a regular quadratic form $(\EE,q,\LL)$ is
an \'etale quadratic $\OO_X$-algebra.  We call its $X$-algebra
isomorphism class in $\Het^1(X,\Z/2\Z)$ the \linedef{discriminant
invariant} $\discs(\EE,q,\LL)$.

\begin{remark}
If 2 is invertible on $X$ and $(\EE,q,\LL)$ is a regular quadratic
form of even rank $n=2m$, then under the canonical homomorphism
$H^1(X,\Z/2\Z) \to H^1(X,\muu_2)$, the discriminant invariant
$d(\EE,q,\LL)$ maps to the class of the signed discriminant module
$\det \EE \tensor (\LL\dual)^{\tensor m}$ (defined
in~\cite[\S4]{parimala_sridharan:norms_and_pfaffians}) of the
associated symmetric bilinear form $b_q$,
cf.~\cite[\S1.9]{auel:clifford}.
\end{remark}

\begin{prop}
\label{prop:disc_perp}
Let $(\EE,q,\LL)$ and $(\EE',q',\LL)$ be regular quadratic forms of
even rank over a scheme $X$.  Then $\discs(q \perp q') = \discs(q) +
\discs(q')$ in $\Het^1(X,\Z/2\Z)$.
\end{prop}
\begin{proof}
We recall (cf. \cite[III~Prop.~4.1.4]{knus:quadratic_hermitian_forms})
that given \'etale quadratic $\OO_X$-algebras $\CliffZ$ and
$\CliffZ'$, the addition of classes $[\CliffZ]$ and $[\CliffZ']$ in
$\Het^1(X,\Z/2\Z)$ is represented by the quadratic \'etale algebra
$\CliffZ \circ \CliffZ'$, defined to be the $\OO_X$-subalgebra of
$\CliffZ \tensor \CliffZ'$ invariant under the diagonal Galois action
$\iota \tensor \iota'$, where $\iota$ and $\iota'$ are the nontrivial
$\OO_X$-automorphisms of $\CliffZ$ and $\CliffZ'$, respectively.

Using
Proposition~\ref{prop:properties_C_1}\ref{properties_C_1.semilinear}, 
we see that restricting the isomorphism \eqref{eq:even_C_0_perp_isom}
to the center yields an $\OO_X$-algebra morphism $\CliffZ(q \perp q')
\to \CliffZ(q) \tensor \CliffZ(q')$, which we claim factors through
$\CliffZ(q) \circ \CliffZ(q')$.  Indeed, for any section $v\tensor v'
\tensor f$ of $\EE_1 \tensor \EE_2$ and $z \tensor z'$ of $\CliffZ(q)
\tensor \CliffZ(q')$, we have that 
$$
(v \tensor v' \tensor f)\, (z \tensor z') = (v \cdot z) \tensor
(v' \cdot z') \tensor f = (\iota(z) * v) \tensor (\iota'(z') *
v') \tensor f = (\iota \tensor \iota')(z\tensor z')\; (v \tensor v' \tensor
f)
$$
by
Proposition~\ref{prop:properties_C_1}\ref{properties_C_1.semilinear},\ref{properties_C_1.m},
where we suppress the canonical embeddings \eqref{eq:C_0_injection},
\eqref{eq:C_1_injection} and the isomorphism
\eqref{eq:even_C_0_perp_isom}.  Hence $(\id - \iota \tensor \iota')
\CliffZ(q \perp q')$ annihilates $\EE \tensor \EE' \tensor \LL\dual$,
hence is zero.  This proves the claim.  The resulting $\OO_X$-algebra
morphism $\CliffZ(q \perp q') \to \CliffZ(q) \circ \CliffZ(q')$ is
Zariski locally an isomorphism by
\cite[IV~\S4.4]{knus:quadratic_hermitian_forms}, hence is an
isomorphism.
\end{proof}

As a consequence, for each line bundle $\LL$, the discriminant
invariant defines a homomorphism $e^1_\LL : I^1(X,\LL) \to
\Hfppf^1(X,\Z/2\Z)$.  By Proposition
\ref{prop:properties}\ref{prop:properties.proj}, $\{e^1_\LL\}$ is a
system of projective similarity class invariants and there is a
\linedef{total discriminant invariant} homomorphism
$$
e^1 : \Itot^1(X) \to \Hfppf^1(X,\Z/2\Z).
$$
Denote by $I^2(X,\LL) \subset I^1(X,\LL)$ the kernel of $e^1_\LL$ and
by $\Itot^2(X) = \oplus_{\LL \in P}I^2(X,\LL)$.  Then $\Itot^2(X)
\subset \ker(e^1)$.

\begin{remark}
The quotient group $\ker(e^1)/\Itot^2(X)$ is generated by elements of
the form $[\EE,q,\LL] - [\EE',q',\LL']$, where both $q$ and $q'$ have
equal \emph{nontrivial} discriminant invariant and yet $\LL$ and
$\LL'$ are in different square classes.  This group will be the
subject of future investigation.
\end{remark}

\subsection{Total Clifford invariant}
\label{subsec:total_Clifford}

For a regular quadratic form $(\EE,q,\LL)$ of even rank $n=2m$ and
trivial discriminant on $X$, the even Clifford algebra decomposes as a
product of Azumaya $\OO_X$-algebras $\CliffAlg_0(\EE,q,\LL) \isom
\CliffAlg_0^+(\EE,q,\LL) \times \CliffAlg_0^-(\EE,q,\LL)$ upon fixing
a splitting idempotent of the center $\CliffZ(\EE,q,\LL) \isom \OO_X
\times \OO_X$.

\begin{prop}
\label{prop:Cliff+-}
Let $X$ be a scheme with 2 invertible and $(\EE,q,\LL)$ be a regular
line bundle-valued quadratic form of rank $n=2m$ and trivial
discriminant.  Then $[\CliffAlg_0^+(\EE,q,\LL)] =
[\CliffAlg_0^-(\EE,q,\LL)]$ in $\Brtwo(X)$.
\end{prop}
\begin{proof}
For $m$ odd, the involution $\tau_0$ is of unitary type with respect
to the center (cf.\ \cite[Prop.~3.11]{auel:clifford}), hence induces
an isomorphism
\begin{equation}
\label{eq:Cliff_op}
\CliffAlg_0^+(\EE,q,\LL) \isom \CliffAlg_0^-(\EE,q,\LL)\op.
\end{equation} 
Hence it suffices to prove that $[\CliffAlg_0^\pm(\EE,q,\LL)]$ are
2-torsion in $\Br(X)$.  For this, we can appeal to the \'etale cohomological Tits algebra
construction of \cite[Thm.~3.17]{auel:clifford}.

For $m$ even, the involution $\tau_0$ is of the first kind and trivial
on the center, restricting to involutions $\tau_0^\pm$ of the first
kind on $\CliffAlg_0^\pm(\EE,q,\LL)$ (in particular, they have
2-torsion Brauer classes).  Thus, there exist refined classes
$[\CliffAlg_0^{\pm}(\EE,q,\LL),\tau_0^{\pm}]$ in $\Het^2(X,\muu_2)$
lifting the Brauer classes $[\CliffAlg_0^\pm(\EE,q,\LL)]$ in
$\Brtwo(X)$ and satisfying
$$
[\CliffAlg_0^{+}(\EE,q,\LL),\tau_0^{+}] +
[\CliffAlg_0^{-}(\EE,q,\LL),\tau_0^{-}] = c_1(\LL,\muu_2),
$$
see~\cite[\S2.8,~\S3.4]{auel:clifford}, where $c_1(\LL,\muu_2) \in
\Het^2(X,\muu_2)$ is the 1st Chern class arising from the coboundary
map of the Kummer squaring sequence.  In particular, we have
$[\CliffAlg_0^+(\EE,q,\LL)] = [\CliffAlg_0^-(\EE,q,\LL)]$ in
$\Brtwo(X)$ since 1st Chern classes are in the kernel of the natural
map $\Het^2(X,\muu_2) \to \Het^2(X,\Gm)$.
\end{proof}

The statement of Proposition~\ref{prop:Cliff+-} (and hence of
Theorem~\ref{thm:Cliff_perp}, below) should remain true without the
hypothesis that 2 is invertible on $X$.  In the setting of
Proposition~\ref{prop:Cliff+-}, we will write
$[\CliffAlg_0^{\pm}(q)] = [\CliffAlg_0^{\pm}(\EE,q,\LL)]$ for the Brauer class in question.

\begin{theorem}
\label{thm:Cliff_perp}
Let $X$ be a scheme with 2 invertible and $(\EE,q,\LL)$ and
$(\EE',q',\LL)$ be regular line bundle-valued quadratic forms of even
rank and trivial discriminant.  Then $[\CliffAlg_0^\pm(q \perp q')] =
[\CliffAlg_0^\pm(q)] + [\CliffAlg_0^\pm(q')]$ in $\Brtwo(X)$.
\end{theorem}
\begin{proof}
Let $e,f$ be complementary central splitting idempotents
of $\CliffAlg_0(q)$, inducing an $\OO_X$-algebra decomposition
$$
\CliffAlg_0(q) = e \CliffAlg_0(q) \times f \CliffAlg_0(q) = \CliffAlg_0^+(q) \times \CliffAlg_0^-(q)
$$
and a corresponding decomposition
$$
\CliffAlg_1(q) = \CliffAlg_1(q) \cdot e \oplus \CliffAlg_1(q) \cdot f
= f \cdot \CliffAlg_1(q) \oplus e \cdot \CliffAlg_1(q)
= \CliffAlg_1^+(q) \oplus \CliffAlg_1^-(q).
$$
making $\CliffAlg_1^\pm(q)$ into a
$\CliffAlg_0^\mp(q)$-$\CliffAlg_0^\pm(q)$-bimodule via the
$\CliffAlg_0(q)$-bimodule structure on $\CliffAlg_1(q)$.  Local
calculations, using
Proposition~\ref{prop:properties_C_1}\ref{properties_C_1.semilinear},
shows that the map in
Propositions~\ref{prop:properties_C_1}\ref{properties_C_1.m} induces
pairings
\begin{equation}
\label{eq:C_1+-.m}
\CliffAlg_1^\pm(q) \times \CliffAlg_1^\mp(q) \to \CliffAlg_0^\mp(q)\tensor\LL.
\end{equation}
Similarly, $\CliffAlg_1^\pm(q)$ annihilates itself via the map in
Propositions~\ref{prop:properties_C_1}\ref{properties_C_1.m},
$\CliffAlg_0^\pm(q)$ and $\CliffAlg_0^\mp(q)$ annihilate each other
via the multiplication in $\CliffAlg_0(q)$, and the
$\CliffAlg_0^\pm(q)$-$\CliffAlg_0^\mp(q)$-bimodule structure on
$\CliffAlg_1^\pm(q)$ induces via the $\CliffAlg_0(q)$-bimodule
structure on $\CliffAlg_1(q)$, is zero.

Let $e',f'$ be complementary central splitting idempotents of
$\CliffAlg_0(q')$, as above.  Then $e \tensor e' + f \tensor f'$ and
$e\tensor f' + f \tensor e'$ (via the isomorphism
\eqref{eq:even_C_0_perp_isom}) are complementary central splitting
idempotents of $\CliffAlg_0(q \perp q')$, inducing a decomposition
$$
\CliffAlg_0(q \perp q') = (e \tensor e' + f \tensor f')
\CliffAlg_0(q\perp q')
\times (e\tensor f' + f \tensor e')\CliffAlg_0(q\perp q') = \CliffAlg_0^+(q\perp q') \times \CliffAlg_0^-(q\perp q').
$$ 
A direct local calculation, using the
$\CliffAlg_0^\mp(q)$-$\CliffAlg_0^\pm(q)$-bimodule structure on
$\CliffAlg_1^\pm(q)$, the pairings \eqref{eq:C_1+-.m}, and the
annihilation statements above, establishes the following block matrix
algebra structures
\begin{align*}
\CliffAlg_0^+(q \perp q') {} & =
\begin{pmatrix}
\CliffAlg_0^+(q)\tensor\CliffAlg_0^+(q') & \CliffAlg_1^-(q)\tensor\CliffAlg_1^-(q')\tensor\LL\dual\\
\CliffAlg_1^+(q)\tensor\CliffAlg_1^+(q')\tensor\LL\dual & \CliffAlg_0^-(q)\tensor\CliffAlg_0^-(q')
\end{pmatrix}\\
\CliffAlg_0^-(q \perp q') {} & =
\begin{pmatrix}
\CliffAlg_0^+(q)\tensor\CliffAlg_0^-(q') & \CliffAlg_1^-(q)\tensor\CliffAlg_1^+(q')\tensor\LL\dual\\
\CliffAlg_1^+(q)\tensor\CliffAlg_1^-(q')\tensor\LL\dual & \CliffAlg_0^-(q)\tensor\CliffAlg_0^+(q') 
\end{pmatrix}
\end{align*}
via the isomorphism \eqref{eq:even_C_0_perp_isom}.  The pairings
\eqref{eq:C_1+-.m} induce morphisms
$$
\CliffAlg_1^\mp(q) \isom \HHom_{\CliffAlg_0^\pm(q)}
\bigl(\CliffAlg_1^\pm(q),\CliffAlg_0^\pm(q)\bigr) \tensor \LL
$$ 
of $\CliffAlg_0^\pm(q)$-$\CliffAlg_0^\mp(q)$-bimodules (these are
right hom sheaves).  Regularity implies that these are isomorphisms,
with respect to which we have $\OO_X$-algebra isomorphisms
\begin{align*}
\CliffAlg_0^+(q \perp q') {} & = 
\EEnd_{\CliffAlg_0^+(q)\tensor\CliffAlg_0^+(q')}\bigl(
\CliffAlg_0^+(q)\tensor\CliffAlg_0^+(q') \oplus
\CliffAlg_1^+(q)\tensor\CliffAlg_1^+(q')\tensor\LL\dual \bigr)\\
\CliffAlg_0^-(q \perp q') {} & = 
\EEnd_{\CliffAlg_0^+(q)\tensor\CliffAlg_0^-(q')}\bigl(
\CliffAlg_0^+(q)\tensor\CliffAlg_0^-(q') \oplus
\CliffAlg_1^+(q)\tensor\CliffAlg_1^-(q')\tensor\LL\dual \bigr).
\end{align*}
In particular, $\CliffAlg_0^+(q \perp q')$ is Brauer equivalent to
$\CliffAlg_0^+(q)\tensor\CliffAlg_0^+(q')$ and $\CliffAlg_0^-(q \perp
q')$ is Brauer equivalent to
$\CliffAlg_0^+(q)\tensor\CliffAlg_0^-(q')$.  An application of
Proposition~\ref{prop:Cliff+-} finishes the proof.
\end{proof}

When 2 is invertible on $X$, then by Theorems~\ref{thm:metabolic} and
\ref{thm:Cliff_perp}, for each line bundle $\LL$ on $X$, the map
$[\EE,q,\LL] \mapsto [\CliffAlg_0^+(\EE,q,\LL)]$ for any choice of
central splitting idempotent, defines a homomorphism $e^2_\LL :
I^2(X,\LL) \to \Brtwo(X)$.  By
Proposition~\ref{prop:properties}\ref{prop:properties.proj}, $\{
e_\LL^2 \}$ is a system of projective similarity class invariants and
there is a \linedef{total Clifford invariant} homomorphism
\begin{equation}
\label{prop:e2}
e^2 : \Itot^2(X) \to \Brtwo(X).
\end{equation}

\begin{remark}
The invariant $e_{\OO_X}^2 : I^2(X) = I^2(X,\OO_X) \to \Brtwo(X)$
coincides with the classical Clifford invariant map.  Indeed, if
$(\EE,q)$ is a regular $\OO_X$-valued quadratic form of even rank and
trivial discriminant then $\CliffAlg_0^+(\EE,q)$ is Brauer equivalent
to the full Clifford algebra $\CliffAlg(\EE,q)$.  See
also~\cite[Thm.~2.10{\it b}]{auel:clifford}.  It was already proved in
\cite{knus_ojanguren:metabolic} that the full Clifford algebra yields
a homomorphism $W(X) \to \Brtwo(X)$.
\end{remark}

\section{Surjectivity of the total Clifford invariant}
\label{sec:surjectivity}

The goal of this section is to prove Theorem~\ref{thm:A}.  Recall that
an Azumaya algebra $\AA$ over a scheme $X$ has $\OO_X$-rank $d^2$ for
a positive integer $d$ called the \linedef{degree}.  The
\linedef{period} of $\AA$ is the order of the Brauer class $[\AA] \in
\Br(X)$.  The \linedef{index} of $\AA$ is the greatest common divisor
of all degrees of Azumaya algebras $\BB$ such that $\AA \tensor
\EEnd\PP \isom \BB \tensor \EEnd\sheaf{Q}$ for vector bundles $\PP$
and $\sheaf{Q}$ on $X$.  If $X$ is integral with function field $K$,
the \linedef{generic index} of $\AA$ is the index of the central
simple $K$-algebra $\AA_K$.  The generic index divides the index, with
equality if $X$ is regular of dimension $\leq 2$.  We will assume that
2 is invertible on $X$.

\subsection{Exceptional isomorphisms}

The exceptional isomorphisms of Dynkin diagrams $\Dynkin{A}_1^2 =
\Dynkin{D}_2$ and $\Dynkin{A}_3 = \Dynkin{D}_3$ have beautiful
reverberations in the theory of quadratic forms of rank 4 and 6,
respectively.  In these ranks, the \linedef{reduced norm} and
\linedef{reduced pfaffian} constructions enable a quadratic form to be
reconstructed from its even Clifford algebra (together with certain
data).  For quadratic forms over rings, this theory was initiated by
Kneser, Knus, Ojanguren, Parimala, Paques, and Sridharan, see
\cite{kneser:composition_binary},
\cite{knus_ojanguren_sridharan:quadratic_azumaya},
\cite{knus_paques:rank_4}, \cite{knus:pfaffians_and_quadratic_forms},
\cite{knus_parimala_sridharan:rank_4},
\cite{knus_parimala_sridharan:rank_6_via_pfaffians}.  Now, a standard
reference on this work is
Knus~\cite[Ch.~V]{knus:quadratic_hermitian_forms}.  Over fields, a
wonderful reference is~\cite[IV~\S15]{book_of_involutions}.  Bichsel
\cite{bichsel:thesis} and Bichsel--Knus
\cite{bichsel_knus:values_line_bundles} provide an extension of this
theory to line bundle-valued forms over rings.  The existing theory
over rings immediately generalizes to base schemes when the
corresponding algebraic groups are of inner type (i.e., the case of
trivial discriminant).  For an approach over general bases using
Severi--Brauer schemes, see
\cite{parimala_sridharan:norms_and_pfaffians}.  In the case of general
discriminant, the details are worked out in \cite[\S5]{auel:clifford}.

We now outline the main results of this theory that we need.  For even
$n=2m$, denote by $\PQF_n^+(X)$ the set of projective similarity
classes of regular line bundle-valued quadratic forms of rank $n$ and
trivial discriminant on $X$.  Denote by $\Aztwo_d(X)$ the set of
isomorphism classes of Azumaya $\OO_X$-algebras of degree $d$ and
period 2.

For ease of exposition, and without loss of generality, we can assume
that $X$ is connected.  The assignment, sending the projective
similarity class of a quadratic form $(\EE,q,\LL)$ of even rank $n=2m$
and trivial discriminant to the unordered pair consisting of the
$\OO_X$-algebra isomorphism classes of the components
$\CliffAlg_0^+(\EE,q,\LL)$ and $\CliffAlg_0^-(\EE,q,\LL)$ of the even
Clifford algebra (for some central splitting idempotent), yields a
well defined map
\begin{equation}
\label{eq:CliffAlg_map}
\PQF_n^+(X) \to \Aztwo_{2^{m-1}}^{(2)}(X)
\end{equation}
where $\{-\}^{(2)}$ denotes the set of unordered pairs of elements.

For any odd $k$, denote by $\Aztwo_{2^k}'(X) \subset
\Aztwo_{2^k}^{(2)}(X)$ the subset of pairs of Brauer equivalent
Azumaya algebras.  For any even $k$, denote by $\Aztwo_{2^k}'(X)$ the
set of equivalence classes of Azumaya algebras of degree $2^k$ and
period 2 under the relation $\AA \sim \BB$ if $\AA \isom \BB$ or $\AA
\isom \BB\op$.  Then for even $k$, there is a canonical injective map
$\Aztwo_{2^k}'(X) \to \Aztwo_{2^k}^{(2)}(X)$ given by $\AA \mapsto
(\AA,\AA\op)$.

For $n \equiv 0 \bmod 4$, recall that $\CliffAlg_0^+(\EE,q,\LL)$ is
Brauer equivalent to $\CliffAlg^-(\EE,q,\LL)$ by
Proposition~\ref{prop:Cliff+-}. 
For $n \equiv 2 \bmod 4$, recall that $\CliffAlg_0^+(\EE,q,\LL)
\isom \CliffAlg_0^-(\EE,q,\LL)\op$ by \eqref{eq:Cliff_op}. 
Hence \eqref{eq:CliffAlg_map} factors through a map
\begin{equation}
\label{eq:CliffAlg_map_refined}
\CliffAlg_0^{\pm} : \PQF_n^+(X) \to \Aztwo_{2^{m-1}}'(X).
\end{equation}
The main result is that for $n = 4$ and $n=6$, the map
\eqref{eq:CliffAlg_map_refined} is a bijection, with inverse map
realized, respectively, by the reduced norm and pfaffian construction
outlined in \cite{knus_parimala_sridharan:rank_6_via_pfaffians},
\cite[V.4--5]{knus:quadratic_hermitian_forms},
\cite{parimala_sridharan:norms_and_pfaffians}.  We now proceed to
summarize these constructions.

\subsubsection*{Reduced norm form}
In the $n=4$ case, given a pair of Brauer equivalent Azumaya
quaternion algebras $\AA$ and $\BB$, fibered Morita theory (cf.\
Lieblich~\cite[\S2.1.4]{lieblich:thesis} or Kashiwara--Schapira
\cite[\S19.5]{kashiwara_schapira:categories_sheaves}) provides a
$\AA$-$\BB$-bimodule $\PP$, which is invertible over $\AA$ and $\BB$
and is unique up to tensoring by a line bundle.  Descending the
reduced norm via \'etale splittings of $\AA$ and $\BB$, there exists a
\linedef{reduced norm form} $\NN(\PP) = (\PP,q_{\PP},\NN_{\PP})$,
consisting of line bundle $\NN_\PP$ and a regular quadratic form
$q_\PP : \PP \to \NN_\PP$ satisfying $q_\PP(a \cdot p \cdot b) =
\Nrd_\AA(a)\, q_{\PP}(p)\, \Nrd_\BB(b)$ for sections $a$ of $\AA$, $b$
of $\BB$, and $p$ of $\PP$, where $\Nrd_\AA : \AA \to \OO_X$ is the
classical reduced norm.  Tensoring $\PP$ by a line bundle induces
a projective similarity of reduced norm forms.  Also, $\PP$ is a
$\BB$-$\AA$-bimodule by composing each action with the standard
involution, giving rise to the same reduced norm form, hence we can
freely exchange the role of $\AA$ and $\BB$.

\subsubsection*{Reduced pfaffian form}
In the $n=6$ case, given an Azumaya algebra $\AA$ of degree 4 and
period 2, there exists a vector bundle $\PP$ of rank 16, unique up to
tensoring by a line bundle, and an $\OO_X$-algebra isomorphism $\vp :
\AA \tensor \AA \isom \EEnd(\PP)$.  The reduced trace, considered as
an element of $\EEnd \AA \isom \AA\op \tensor \AA$, is mapped via
$\vp$ to an involutory $\OO_X$-module endomorphism $\psi : \PP \to
\PP$.  The subsheaf $A_\psi(\PP) = \im(\id_\PP - \psi)$ of alternating
elements with respect to $\psi$ is a vector bundle of rank $6$, as can
be checked \'etale locally.  Descending the pfaffian map via \'etale
splitting of $\AA$ and $\PP$, there exists a \linedef{reduced pfaffian
form} $\Pf(\PP) = (A_\psi(\PP),\pf_\PP,\Pf_\PP)$, consisting of a line
bundle $\Pf_\PP$ and a regular quadratic form $\pf_\PP : A_\psi(\PP)
\to \Pf_\PP$. Tensoring $\PP$ by a line bundle tensors $A_\psi(\PP)$
by the square of the line bundle, inducing a projective similarity of
reduced pfaffian forms.  Exchanging $\AA$ with $\AA\op$ replaces $\PP$
by $\PP\dual$ and $\psi$ by $\psi\dual$, giving rise to isomorphisms
$A_{\psi\dual}(\PP\dual) \isom A_\psi(\PP)\dual$ and $\Pf_{\PP\dual}
\isom (\Pf_\PP)\dual$ (cf.\
\cite[III~Lemma~9.3.5]{knus:quadratic_hermitian_forms}) and a
projective similarity of reduced pfaffian forms $\Pf(\PP)$ and
$\Pf(\PP\dual)$ (cf.\
\cite[III~Prop.~9.4.2]{knus:quadratic_hermitian_forms}).

\begin{theorem}
\label{thm:low_dim}
Let $X$ be a scheme with 2 invertible.
\begin{enumerate}
\item \label{thm:low_dim.4}There are inverse bijections
$$
\xymatrix@C=40pt{
\PQF_4^+(X) \ar@<1mm>[r]^(.5){\CliffAlg_0^{\pm}} & 
\ar@<1mm>[l]^(.52){\NN} \Aztwo_2'(X)
}
$$
where $\NN$ is the reduced norm form construction.

\item \label{thm:low_dim.6} There are inverse bijections 
$$
\xymatrix@C=40pt{
\PQF_6^\pm(X) \ar@<1mm>[r]^(.54){\CliffAlg_0^+} & 
\ar@<1mm>[l]^(.48){\Pf} \Aztwo_4'(X)
}
$$
where $\Pf$ is the reduced pfaffian form construction.
\end{enumerate}
\end{theorem}
\begin{proof}
Given Brauer equivalent Azumaya algebras $\AA$ and $\BB$, there exists
an invertible $\AA$-$\BB$-bimodule $\PP$ such that $\BB \isom
\EEnd_\AA(\PP)$. 
Hence for $m$ even (e.g., $m=2$), $\Aztwo_{2^{m-1}}'(X)$ is in
bijection with the set of isomorphism classes of pairs $(\AA,\PP)$,
consisting of an Azumaya algebra $\AA$ of degree $n$ and an invertible
right $\AA$-module $\PP$.  A direct proof of \ref{thm:low_dim.4} can
be deduced
from~\cite[Prop.~4.1]{parimala_sridharan:norms_and_pfaffians} (itself
a generalization
of~\cite[Prop.~4.5]{bichsel_knus:values_line_bundles}), which states
that if $\AA$ is an Azumaya quaternion algebra and $\PP$ is an
invertible right $\AA$-module, then $\CliffAlg_0(\NN(\PP)) \isom \AA
\times \EEnd_{\AA}(\PP)$.  By
\cite[Prop.~4.3]{parimala_sridharan:norms_and_pfaffians}, the map
$\NN$ is surjective.  Hence $\NN$ and $\CliffAlg_0^{\pm}$ are inverse
bijections.  This is a generalization
of~\cite[Thm.~10.7]{knus_parimala_sridharan:rank_6_via_pfaffians} to
the line bundle-valued (trivial discriminant) setting.

A direct proof of \ref{thm:low_dim.6} can be given along similar
lines. By~\cite[Prop.~4.8]{bichsel_knus:values_line_bundles} (which
immediately generalizes to general base schemes), if $\AA$ is an
Azumaya $\OO_X$-algebra of degree 4, $\PP$ is a locally free
$\OO_X$-module of rank 16, and $\vp : \AA \tensor \AA \to \EEnd(\PP)$
is an $\OO_X$-algebra isomorphism (corresponding to the element $[\AA]
\in \Aztwo_4'(X)$), then $\CliffAlg_0(\Pf(\PP)) \isom \AA\op \times
\EEnd_{\AA\op}(\PP)$.  By
\cite[Prop.~6.1]{parimala_sridharan:norms_and_pfaffians}, the map
$\Pf$ is surjective.  Hence $\Pf$ and $\CliffAlg_0^\pm$ are inverse
maps.  This is a generalization
of~\cite[Thm.~9.4]{knus_parimala_sridharan:rank_6_via_pfaffians} to
the line bundle-valued (trivial discriminant) setting.
\end{proof}

As a result, we can realize any Azumaya algebra of degree dividing 4
on $X$ as the even Clifford invariant of a line bundle-valued
quadratic form.  In particular, if $\Brtwo(X)$ is generated by such
Azumaya algebras, then the total Clifford invariant is surjective.

\begin{cor}
\label{thm:main1}
Let $X$ be a scheme with 2 invertible.  
If $\Brtwo(X)$ is generated by Azumaya algebras of degree $\leq
4$, then the total Clifford invariant
$$
e^2 : \Itot^2(X) \to \Brtwo(X)
$$ 
is surjective. 
\end{cor}

Note that if $X$ is the spectrum of a field, then $\Brtwo(X)$ is always
generated by quaternion algebras by Merkurjev's theorem, hence the
hypotheses of Corollary~\ref{thm:main1} are quite global in nature.

In the same spirit, we can give a stronger condition sufficient for
the surjectivity of the classical Clifford invariant $e^2_{\OO_X} :
I^2(X) \to \Brtwo(X)$.  First we recall some results from
\cite{knus_parimala_sridharan:rank_6_via_pfaffians}.  Let $[\AA] \in
\Aztwo_4(X)$ have reduced pfaffian form
$(A_\psi(\PP),\pf_\PP,\Pf_\PP)$, choosing a vector bundle $\PP$ of
rank 16 such that $\AA\tensor\AA \isom \EEnd\PP$.  The class
$d_0(\AA)=[\Pf_\PP] \in \Pic(X)/2$ is a well defined invariant of
$\AA$,
see~\cite[\S9,~p.~213]{knus_parimala_sridharan:rank_6_via_pfaffians}.
When $d_0(\AA)$ is trivial we say that $\AA$ has \linedef{trivial
pfaffian invariant}.  

\begin{prop}[{\cite[Prop.~3.2]{knus_parimala_sridharan:rank_6_via_pfaffians}}]
\label{prop:inv}
Let $X$ be a scheme with 2 invertible and $\AA \in \Aztwo_4(X)$.  If
$\AA$ has an involution of the first kind then $d_0(\AA)$ is
2-torsion.  Moreover, if $\AA$ has a symplectic involution then
$d_0(\AA)$ is trivial.
\end{prop}

We recall that any Azumaya quaternion algebra has a standard
symplectic involution, hence has trivial pfaffian invariant.

\begin{cor}
\label{thm:main2}
Let $X$ be a scheme with 2 invertible.  If $\Brtwo(X)$ is generated by
Azumaya algebras of degree dividing 4 with trivial pfaffian invariant,
then the classical Clifford invariant
$$
e_{\OO_X}^2 : I^2(X) \to \Brtwo(X)
$$ 
is surjective.  In particular this is the case if $\Brtwo(X)$ is
generated by Azumaya quaternion algebras.
\end{cor}
\begin{proof}
We first remark that any $\AA \in \Aztwo_4(X)$ of index 2 is Brauer
equivalent to $\AA' \in \Aztwo_4(X)$ with trivial pfaffian invariant.
Indeed, if $\AA$ has index 2, then $\AA \isom \EEnd_{\BB}(\PP)$ for an
Azumaya quaternion algebra $\BB$ and a locally free $\BB$-module $\PP$
of rank 2.  
We can extend the standard symplectic involution on $\BB$ to $\AA' =
\MMat_2(\BB)$, which then has trivial pfaffian invariant by
Proposition \ref{prop:inv}.  But $\AA$ is Brauer equivalent to $\AA'$.

Now, note that the reduced norm form $q_{\BB} : \BB \to \OO_X$ is a
regular $\OO_X$-valued quadratic form in $I^2(X)$ with
$e^2_{\OO_X}(\NN(\BB)) = [\BB]$, by Theorem
\ref{thm:low_dim}\ref{thm:low_dim.4}.  This already proves the final
claim.  In general, if $\AA \in \Aztwo_4(X)$ has trivial pfaffian
invariant, then there exists an $\OO_X$-valued quadratic form
$(\EE,q)$ in the projective similarity class of $\Pf(\AA)$.  By
Theorem \ref{thm:low_dim}\ref{thm:low_dim.6}, we have that
$e^2_{\OO_X}(\EE,q) = [\AA]$. The first claim follows.
\end{proof}

\subsection{Brauer dimension results}

Now we investigate sufficient conditions under which $\Brtwo(X)$
is generated by Azumaya algebras of degree dividing 4.  Let $X$ be an
integral scheme with function field $K$.  An Azumaya $\OO_X$-algebra
$\AA$ is called an \linedef{Azumaya division algebra} if the generic
fiber $\AA_K$ is a central division $K$-algebra.  

We introduce two conditions on an integral scheme $X$
with function field $K$:
\begin{enumerate}\setlength{\itemsep}{2mm}
\item[\textit{A}] Every central division $K$-algebra of period 2 and
degree dividing $4$, which is Brauer equivalent to the generic fiber
of an Azumaya $\OO_X$-algebra, is isomorphic to the generic fiber of
an Azumaya division $\OO_X$-algebra, i.e., restriction to the generic point $\Aztwo_d(X) \to
\Aztwo_d(K)$ is surjective for $d$ dividing $4$. 

\item[\textit{B}] Every $\AA \in \Brtwo(X)$ satisfies $\ind(\AA_K)\, |\,
\per(\AA_K)^2$, i.e., $\ind(\AA_K)\, |\, 4$.
\end{enumerate}

Condition \textit{A} is a kind of ``purity for division algebras'' of
period 2 and degree dividing 4, or ``purity for
$\GL_4/\mu_2$-torsors'' in the setting of Colliot-Th\'el\`ene--Sansuc~\cite{colliot-thelene_sansuc:fibres_quadratiques}.  Condition \textit{B} might be
restated loosely as ``$X$ has Brauer dimension 2'' for classes of
period 2.  See~\cite[\S4]{ketura} for the precise notion of Brauer
dimension.

We now prove that under conditions \textit{A} and \textit{B}, we get
an ``unramified symbol length'' bound on the Brauer group, which is
stronger than the generation hypothesis needed for
Corollary~\ref{thm:main1}.

\begin{theorem}
\label{prop:main2}
Let $X$ be a regular integral scheme with 2 invertible.  If $X$
satisfies conditions \textit{A} and \textit{B}, then $\Brtwo(X)$ is
represented by Azumaya algebras of degree dividing~4.  In particular,
the total Clifford invariant is surjective.
\end{theorem}
\begin{proof}
Since $X$ is regular, the canonical map $\Br(X) \to \Br(K)$ is
injective, see~\cite{auslander_goldman:brauer_group_commutative_ring}
or \cite[Cor.~1.8]{grothendieck:Brauer_II}.  By condition \textit{B},
for any $\AA \in \Brtwo(X)$, we have that $\AA_K \in \Brtwo(K)$ is
Brauer equivalent to a central division $K$-algebra $D$ of degree
dividing 4.  By condition \textit{A}, there exists an Azumaya
$\OO_X$-algebra $\BB$ whose generic fiber is $D$, in particular, $\BB$
has degree dividing $4$.  Since $[\BB_K] = [D] = [\AA_K] \in
\Brtwo(K)$, by the injectivity of $\Br(X) \to \Br(K)$, we have that
$[\BB] = [\AA] \in \Brtwo(X)$.   
The final claim is thus a direct consequence of Corollary~\ref{thm:main1}.
\end{proof}

We now collect together some necessary conditions under which
conditions \textit{A} and \textit{B} hold.  Condition \textit{A} (and
more generally, purity for division algebras of any degree) is
satisfied quite generally for schemes of dimension $\leq 2$.

\begin{theorem}
\label{thm:purity}
Any regular integral scheme $X$ of dimension $\leq 2$ satisfies
condition \textit{A}.
\end{theorem}
\begin{proof}
Apply
Colliot-Th\'el\`ene--Sansuc~\cite[Cor.~6.14]{colliot-thelene_sansuc:fibres_quadratiques}
to the reductive group scheme $\GL_4/\mu_2$ over $X$.  
An alternate proof can be found
in~\cite[Thm.~4.3]{auel_parimala_suresh:degquadsurface}.
\end{proof}

Note that for schemes of higher dimension, Condition \textit{A} can
fail, see \cite{antieau_williams:orders}

As for condition \textit{B}, it holds in the following cases where the
Brauer dimension of $K$ is known to be 1:
\begin{itemize}
\item smooth curves over finite fields (by class field theory),
\item smooth surfaces over algebraically closed fields (by
Artin~\cite{artin:Brauer-Severi} or de Jong~\cite{dejong:surfaces}),
\end{itemize}
and where the Brauer dimension of $K$ is known to be 2:
\begin{itemize}
\item smooth curves over local fields (by Saltman~\cite{saltman:division_algebra_p-adic_curves}),
\item smooth surfaces over (pseudo-)finite fields (by Lieblich~\cite{lieblich:transcendence_2}).
\end{itemize}
We can now proceed to prove Theorem~\ref{thm:A}.

\begin{cor}
\label{cor:main_result}
Let $X$ be regular integral scheme with 2 invertible. 
\begin{enumerate}
\item \label{cor:main_result.1} If $X$ is a smooth curve over a finite field or surface over an
algebraically closed field, then the classical Clifford invariant $e^2
: I^2(X) \to \Brtwo(X)$ is surjective.

\item \label{cor:main_result.2} If $X$ is a smooth curve over a local field or a surface over a
(pseudo-)finite field, then the total Clifford invariant $e^2 :
\Itot^2(X) \to \Brtwo(X)$ is surjective.
\end{enumerate}
\end{cor}
\begin{proof}
This is a direct consequence of Corollaries \ref{thm:main1} and
\ref{thm:main2}, Theorem~\ref{prop:main2}, Theorem \ref{thm:purity}, and the
Brauer dimension results stated above.  Note that \ref{cor:main_result.1} was
already known for curves over finite fields by~\cite[Lemma~4.1]{parimala_sridharan:graded_Witt} and for surfaces over
algebraically closed fields by
\cite{fernandez-carmena:Witt_group_surfaces}.
\end{proof}

We remark that recent results of
Lieblich--Parimala--Suresh~\cite{lieblich_parimala_suresh:u-invariant}
imply that, assuming a conjecture of Colliot-Th\'el\`ene on the
Brauer--Manin obstruction to the existence of 0-cycles of degree 1 on
smooth projective varieties over global fields, condition \textit{B}
also holds for regular arithmetic surfaces, i.e., regular schemes
proper and flat over the spectrum of the ring of integers of a number
field whose generic fiber is a geometrically connected curve.  Thus
Theorem~\ref{thm:A} holds conditionally for regular arithmetic
surfaces. Also, recent results of Harbater--Hartmann--Krashen
\cite{HHK:refinements} prove condition \textit{B} for a wide class of
local curves over complete discrete valuation rings with finite or
algebraically closed residue fields.

\subsection{A total unramified Milnor question}

We are lead to the following natural question, inspired by our main result.

\begin{question}
\label{total_question}
Let $X$ be a regular integral scheme with 2 invertible.  Assume that
the function field $K$ of $X$ satisfies $\cd_2(K) \leq 3$.  Is the
homomorphism
$$
e^2 : \Itot^2(X) \to \Brtwo(X)
$$
surjective?
\end{question}

A positive answer to Question \ref{total_question} brings a scheme
closer to having a positive answer to an analogue of the unramified
Milnor question for the fundamental filtration $\Itot^2(X) \subset
\Itot^1(X) \subset \Wtot(X)$ of the total Witt group; see
\cite[Question~3.1]{auel:kyoto} for a survey of results on the
unramified Milnor question.   All schemes
appearing in Corollary~\ref{cor:main_result} have a positive answer to
Question \ref{total_question}.

There are recent examples of
Antieau--Williams~\cite[\S7]{antieau_williams:topological_period-index_6-complexes},~\cite[Example~3.13]{antieau_williams:topology_purity_torsors}
of smooth affine schemes over $\C$ of dimension 5 with nonsurjective total Clifford
invariant (these examples actually have nonsurjective classical
Clifford invariant and trivial Picard group).

\providecommand{\href}[2]{#2}

\end{document}